\documentclass[12pt]{article}
\usepackage{amsmath,amssymb,amsthm,eucal, mathtools, mathrsfs}
\usepackage[T1]{fontenc}
\usepackage{color}
\usepackage[all]{xy}
\usepackage{slashed}
\usepackage{enumitem}
\usepackage[pdftex, draft=false]{hyperref} 
\usepackage{leftindex}
\title{\vspace*{-1pc}%
Curvature and Weitzenbock formula for spectral triples}

\author{Bram Mesland\S$^*$, Adam Rennie\dag
\thanks{email: 
\texttt{b.mesland@math.leidenuniv.nl}, \texttt{renniea@uow.edu.au}
}
\\[3pt]
\S Mathematisch Instituut, Universiteit Leiden, Netherlands
\\[3pt]
\dag School of Mathematics and Applied Statistics, University of Wollongong\\
Wollongong, Australia
}

\advance\topmargin by -\headheight
\advance\topmargin by -\headsep
\evensidemargin=\oddsidemargin
\makeatletter
\def\section{\@startsection{section}{1}{\z@}{-3.5ex plus -1ex minus
  -.2ex}{2.3ex plus .2ex}{\large\bf}}
\def\subsection{\@startsection{subsection}{2}{\z@}{-3.25ex plus -1ex
  minus -.2ex}{1.5ex plus .2ex}{\normalsize\bf}}
\makeatother

\numberwithin{equation}{section} 

\theoremstyle{plain} 
\newtheorem{thm}{Theorem}[section]
\newtheorem{lemma}[thm]{Lemma}
\newtheorem{prop}[thm]{Proposition}
\newtheorem{corl}[thm]{Corollary}

\theoremstyle{definition} 
\newtheorem{defn}[thm]{Definition}
\newtheorem*{ass*}{Standing assumption}
\newtheorem{example}[thm]{Example} 
\theoremstyle{remark} 
\newtheorem{rmk}[thm]{Remark}

\DeclareMathOperator{\End}{End}   
\DeclareMathOperator{\Hom}{Hom}   

\newcommand{\nablar}{\overrightarrow{\nabla}} 
\newcommand{\nablal}{\overleftarrow{\nabla}} 

\newcommand{\B}{\mathcal{B}}  
\newcommand{\C}{\mathbb{C}}   

\newcommand{\D}{\mathcal{D}}  
\renewcommand{\d}{\mathrm{d}_\Psi} 
\newcommand{\dee}{\mathrm{d}} 
\renewcommand{\H}{\mathcal{H}}  
\newcommand{\ox}{\otimes}     
\newcommand{\X}{\mathcal{X}}  
 
\newcommand{\Y}{\mathcal{Y}} 
\newcommand{\Z}{\mathbb{Z}}   

\newcommand{\bra}[1]{\langle#1|} 
\newcommand{\pairing}[2]{\langle #1\mathbin{|}#2\rangle} 
\newcommand{\stroke}{\mathbin|}     
\newcommand{\alphar}{\overrightarrow{\alpha}}
\newcommand{\alphal}{\overleftarrow{\alpha}}

\def\pairL_#1(#2|#3){{}_{#1}(#2\stroke#3)} 
\def\pairR(#1|#2)_#3{(#1\stroke#2)_{#3}} 
\def\scal<#1|#2>{\langle#1\stroke#2\rangle} 

\newbox\ncintdbox \newbox\ncinttbox 
	\setbox0=\hbox{$-$}
	\setbox2=\hbox{$\displaystyle\int$}
	\setbox\ncintdbox=\hbox{\rlap{\hbox
		to \wd2{\hskip-.125em \box2\relax\hfil}}\box0\kern.1em}
	\setbox0=\hbox{$\vcenter{\hrule width 4pt}$}
	\setbox2=\hbox{$\textstyle\int$}
	\setbox\ncinttbox=\hbox{\rlap{\hbox
		to \wd2{\hskip-.175em \box2\relax\hfil}}\box0\kern.1em}

\allowdisplaybreaks
\begin{document}

\maketitle

\vspace{-2pc}

\begin{abstract}
Using the Levi-Civita connection on the noncommutative differential one-forms of a spectral triple $(\B,\H,\D)$, we define the full Riemann curvature tensor, the Ricci curvature tensor and scalar curvature. We give a definition of Dirac spectral triples and  derive a general Weitzenbock formula for them.
We apply these tools to $\theta$-deformations of compact Riemannian manifolds. We show that the Riemann and Ricci tensors transform naturally under $\theta$-deformation, whereas the connection Laplacian, Clifford representation of the curvature and the scalar curvature are all invariant under deformation.
\end{abstract}

\tableofcontents
\parskip=6pt
\parindent=0pt

\section{Introduction}
\label{sec:intro}

\vspace{-8pt}

This article presents 
the curvature tensor
for spectral triples using the formalism of \cite{MRLC}.
For connections on the module of one-forms we also define Ricci and scalar curvature.
We then  prove a general Weitzenbock formula for Dirac 
spectral triples, and exemplify this by establishing the formula for $\theta$-deformations of commutative manifolds.

One recent approach to curvature in noncommutative geometry is via heat kernel coefficients, \cite{CT11,CM14}.
We do not pursue this approach, rather we adapt the long standing algebraic definitions to the context of spectral triples. In particular we exploit our construction of the Levi-Civita connection \cite{MRLC} to define a preferred curvature.

The curvature tensors we present are  concrete operators computed as $\nabla^2$ and contractions thereof, familiar from differential geometry and algebra, see \cite{BMBook} and references therein. Calculations of these curvatures is of comparable difficulty to the manifold case, so that for situations with reasonable symmetry they can be done by hand.

To relate the operator of a spectral triple to the curvature, we introduce the class of Dirac spectral triples, emulating the notion of Dirac bundle on a manifold. In this setting we can define connection Laplacians and obtain a Weitzenbock formula.
The positivity of connection Laplacians relies on the vanishing of a divergence term, just as in the manifold case.

The formulation of a noncommutative Weitzenbock formula requires the existence of a braiding on the module of two-tensors. On a manifold, the flip map plays the role of the braiding. 
To justify our formula, we explain this issue in detail in Section \ref{subsec:in-detail}. 

While the flip map is typically not well-defined on noncommutative tensor products,
there are numerous examples of braidings in the algebraic context \cite{BMBook}. We provide 
examples of 
braidings for 
$\theta$-deformations \cite{MRLC}, and the Podle\'s sphere \cite{MRPods}.
In \cite{MRLC}  braidings appeared for a related reason, and were used to obtain reality conditions on two-tensors and uniqueness of Hermitian and torsion-free bimodule connections on the module of one-forms.

In \cite{MRLC} we constructed the unique Levi-Civita connection for $\theta$-deformations. Here we show that the Levi-Civita connection of the $\theta$-deformed manifold coincides with the $\theta$-deformation of the Levi-Civita connection of the manifold. 
This allows us to establish the Weitzenbock formula and show that the scalar curvature  remains undeformed, while the full curvature tensor and Ricci tensor transform naturally under deformation.

Section 2 recalls the framework of \cite{MRLC}, Section 3 discusses curvature tensors, Section 4 introduces Dirac modules and proves a general Weitzenbock formula. Section 5 presents the example of $\theta$-deformations.

{\bf Acknowledgements} The authors thank the Erwin Schr\"{o}dinger 
Institute, 
Austria, for hospitality and support during the 
early stages of this work. 
BM thanks the University of Wollongong, and AR thanks the University of Leiden for hospitality.
We thank Alan Carey, Giovanni Landi and Walter van Suijlekom for important 
discussions.

\section{Background on noncommutative forms and connections}
\label{sec:ass1}
This section sets notation and summarises the setup needed to obtain a (unique) Hermitian torsion-free connection on the module of one-forms of a spectral triple. We do this by reformulating the assumptions and results of \cite{MRLC} in the context of spectral triples.
\subsection{Modules of forms}
Throughout this article we are looking at the differential structure provided by a spectral triple. We never require a compact resolvent condition, and so have omitted it from the following, otherwise standard, definition.
See \cite[Examples 2.4--2.6]{MRLC}.
\begin{defn}
\label{defn:NDS}
Let $B$ be a $C^{*}$-algebra. A spectral triple for $B$ is a triple $(\B,\H,\D)$ where $\B\subset B$ is a local \cite[Definition 2.1]{MRLC} dense $*$-subalgebra, $\H$ is a  Hilbert space equipped with a $*$-representation $B\to \mathbb{B}(\H)$, and $\D$ an unbounded self-adjoint regular operator $\D:\,{\rm dom}(\D)\subset \H\to \H$ such that for all $a\in\B$
\[
a\cdot{\rm dom}(\D)\subset{\rm dom}(\D)\quad \mbox{and} \quad [\D,a]\quad\mbox{is bounded}.
\]
\end{defn}

Given a spectral triple $(\B,\H,\D)$, the module of one-forms is the space
\[
\Omega^{1}_{\D}(\B):=\textnormal{span}\left\{a[\D,b]:a,b\in\B\right\}\subset\mathbb{B}(\H).
\]
We obtain a first order differential calculus $\mathrm{d}:\B\to \Omega^{1}_{\D}(\B)$ by setting $\mathrm{d}(b):=[\D,b]$. This calculus carries an involution $(a[\D,b])^{\dag}:=[\D,b]^{*}a^{*}$ induced by the operator adjoint. Thus $(\Omega^{1}_{\D}(\B),\dag)$ is a first order differential structure in the sense of   \cite{MRLC}.

We recollect some of the constructions of \cite{MRLC} for $(\Omega^{1}_{\D}(\B),\dag)$. Writing $T^{k}_{\D}(\B):=\Omega_\D^1(\B)^{\ox k}$, the universal differential
forms $\Omega^*_{u}(\B)$ admit a representation
\begin{align}
\widehat{\pi}_\D:\Omega^k_{u}(\B)\to T^{k}_{\D}(\B)\quad \widehat{\pi}_\D(a_0\delta(a_1)\cdots\delta(a_k))&=a_0[\D,a_1]\ox\cdots\ox[\D,a_k],\\
\pi_\D:=m\circ\widehat{\pi}_{\D}:\Omega^k_{u}(\B)\to \Omega^k_\D(\B)\quad \pi_\D(a_0\delta(a_1)\cdots\delta(a_k))&=a_0[\D,a_1]\cdots[\D,a_k],
\end{align}
where $m:T^{k}_{\D}(\B)\to\Omega^k_\D(\B)$ is the multiplication map.
Neither $\pi_\D$ nor $\widehat{\pi}_\D$ are maps of differential algebras, but are $\B$-bilinear maps of associative $*$-$\B$-algebras, \cite{Landi,MRLC}. The $*$-structure on $\Omega_\D^*(\B)$ is determined by the adjoint of linear maps on $\H$, while the $*$-structure on $\oplus_kT^{k}_{\D}(\B)$ is given by the operator adjoint and
\[
(\omega_1\ox\omega_2\ox\cdots\ox\omega_k)^\dag:= 
\omega_k^*\ox_\B\cdots\ox\omega_2^*\ox_\B\omega_1^*.
\]
We will write $\omega^\dag:=\omega^*$ for one forms $\omega$ as well, though we will adapt this notation when we come to $\theta$-deformations. 
 
The maps $\widehat{\pi}:\Omega^{*}_{u}(\B)\to T^*_\D$ and $\delta:\Omega^{k}_{u}(\B)\to \Omega^{k+1}_{u}(\B)$ are typically not compatible  in the sense that $\delta$ need not map $\ker \widehat{\pi}$ to itself. Thus in general,  $T^{*}_{\D}(\B)$ can not be made into a differential algebra.
The issue to address is that
there are universal
forms $\omega\in \Omega^n_u(\B)$ for which $\widehat{\pi}(\omega)=0$
but $\widehat{\pi}(\delta(\omega))\neq0$. The latter are known as {\em junk forms},  \cite[Chapter VI]{BRB}. 
We denote the $\B$-bimodules of junk forms by
\begin{align*}
J^k_\D(\B)&=\{\pi_\D(\delta(\omega)):\,\pi_\D(\omega)=0\}\quad\mbox{and}\quad
JT^k_\D(\B)=\{\widehat{\pi}_\D(\delta(\omega)):\,\widehat{\pi}_\D(\omega)=0\}.
\end{align*}
Observe that the junk submodules only depend on the representation of the universal forms.
\begin{defn}
\label{ass:fgp-metric-junk}
A second order differential structure $(\Omega^{1}_{\D},\dag,\Psi)$ is a first order differential structure $(\Omega^{1}_{\D}(\B),\dag)$ together with an idempotent $\Psi:T^{2}_{\D}\to T^{2}_{\D}$ satisfying $\Psi\circ\dag=\dag\circ\Psi$ and $JT^2_\D(\B)\subset{\rm Im}(\Psi)\subset m^{-1}(J^2_\D(\B))$. A second order differential structure is Hermitian if
$\Omega^1_\D(\B)$ is a finitely generated projective right $\B$-module with right inner product $\pairing{\cdot}{\cdot}_\B$, such that $\Psi=\Psi^{2}=\Psi^{*}$ is a projection. 
\end{defn}

A second order differential structure admits an exterior derivative $\d:\Omega^1_\D(\B)\to T^2_\D(\B)$ 
via 
\begin{equation}
\label{eq: second-order-diff}
\d(\rho)=(1-\Psi)\circ \widehat{\pi}_\D\circ\delta\circ \pi_\D^{-1}(\rho).
\end{equation} 
The differential satisfies $\d([\D,b])=0$ for all $b\in\B$. A differential on one-forms allows us to define curvature for modules, and formulate torsion for connections on one-forms. 

For an Hermitian differential structure $(\Omega^{1}_{\D}(\B),\dag,\Psi,\pairing{\cdot}{\cdot})$, the module of one-forms $\Omega^1_\D(\B)$ is also a finite projective left module \cite[Lemma 2.12]{MRLC} with inner product ${}_\B\pairing{\omega}{\rho}=\pairing{\omega^\dag}{\rho^\dag}_\B$. Thus all tensor powers $T^{k}_{\D}(\B)$ carry right and left inner products. Using these we obtain bimodule isomorphisms
\begin{align}
\alphar&:T^{n+k}_{\D}\to \overrightarrow{\textnormal{Hom}}^{*}_{\B}(T^{k}_{\D},T^{n}_{\D}), \quad\alphar(\omega\otimes \eta)(\rho):=\omega\pairing{\eta^{\dag}}{\rho}\nonumber\\
\alphal&:T^{n+k}_{\D}\to \overleftarrow{\textnormal{Hom}}^{*}_{\B}(T^{k}_{\D},T^{n}_{\D}), \quad\alphal( \eta\otimes \omega)(\rho):=\pairing{\rho^{\dag}}{\eta}\omega,
\label{eq:alphas}
\end{align}
where $\rho,\eta\in T^{k}_{\D}$ and $\omega\in T^{n}_{\D}$.
Inner products on $\Omega^k_\D(\B)$ do not arise automatically.

The two inner products on $\Omega^1_\D(\B)$ give rise to equivalent norms on $\Omega^1_\D(\B)$, and using results of \cite{KajPinWat}, $\Omega^1_\D(\B)$ is a bi-Hilbertian bimodule of finite index. To explain what this means for us, recall \cite{FL02} that a (right) frame for $\Omega^1_\D(\B)$ is a (finite) collection of elements $(\omega_j)$ that satisfy
\[
\rho=\sum_j\omega_j\pairing{\omega_j}{\rho}_\B
\]
for all $\rho\in \Omega^1_\D(\B)$. 
A finite projective bi-Hilbertian module has a  ``line element'' or ``quantum metric'' \cite{BMBook} given by 
\begin{equation}
G=\sum_j\omega_j\ox\omega_j^\dag.
\label{eq:Gee}
\end{equation} 
The line element $G$ is independent of the choice of frame, is central, meaning that $bG=Gb$ for all $b\in\B$, and
\[
{\rm span}_\B\left\{\sum_j\omega_j\ox\omega_j^\dag:\,\mbox{for any frame }(\omega_j)\right\}
\]
is a complemented submodule of $T^{2}_{\D}$. The endomorphisms $\alphar(G)$ and $\alphal(G)$ coincide with the identity operator on $\Omega^{1}_{\D}(\B)$. The inner product is computed  via 
\begin{equation}
-g(\omega\ox\rho):=\pairing{G}{\omega\ox\rho}_\B=\pairing{\omega^\dag}{\rho}_\B.
\label{eq:Riemann}
\end{equation}
Such bilinear inner products appear in \cite{BMBook,BGJ2,BGJ1}. The element 
\begin{equation}
e^\beta:=\sum_j{}_\B\pairing{\omega_j}{\omega_j}=-g(G)\in\B
\label{eq:ee-beta}
\end{equation} 
is independent of the choice of right frame, and
is central, positive and invertible (provided the left action of $\B$ on $\Omega^1_\D(\B)$ is faithful). Setting $Z=e^{-\beta/2}\sum_j\omega_j\ox\omega_j^\dag$, the endomorphisms $\alphar(Z\otimes Z)$ and $\alphal(Z\otimes Z)$ of $T^{2}_\D(\B)$ are projections. 
\subsection{Existence of Hermitian torsion-free connections}
A right connection on a right $\B$-module $\X$ is a $\C$-linear map
\[
\nablar:\X\to \X\otimes_{\B}\Omega^{1}_{\D},\quad\mbox{such that}\quad
\overrightarrow{\nabla}(x a)=\overrightarrow{\nabla}(x)a+x\ox[\D,a].
\]
There is a similar definition for left connections on left modules.
Connections always exist on finite projective modules. Given a connection $\overrightarrow{\nabla}$ on a right inner product $\B$-module $\X$ we say that $\overrightarrow{\nabla}$ is Hermitian \cite[Definition 2.23]{MRLC} if
for all $x,y\in\X$ we have
\[
-\pairing{\overrightarrow{\nabla}x}{y}_\B+\pairing{x}{\overrightarrow{\nabla}y}_\B=[\D,\pairing{x}{y}_\B].
\]
For left connections we instead require
\[
{}_\B\pairing{\overleftarrow{\nabla}x}{y}-{}_\B\pairing{x}{\overleftarrow{\nabla}y}=[\D,{}_\B\pairing{x}{y}].
\]
If furthermore $\X$ is a $\dag$-bimodule \cite[Definition 2.8]{MRLC} like $T^k_\D$, then for each right connection $\overrightarrow{\nabla}$ on $\X$ there is a conjugate left connection $\overleftarrow{\nabla}$ given by
$\overleftarrow{\nabla}=-\dag\circ\overrightarrow{\nabla}\circ\dag$ which is Hermitian if and only if $\overrightarrow{\nabla}$ is Hermitian.
\begin{example}

Given a (right) frame $v=(x_j)\subset \X$ we get a left- and right Grassmann connection
\[
\nablal^{v}(x):=[\D,{}_\B\pairing{x}{x_{j}^{\dag}}]\otimes x_{j}^{\dag},\quad \nablar^{v}(x):=x_j\ox[\D,\pairing{x_j}{x}_\B],\qquad x\in\X.
\]
The Grassmann connections are Hermitian and conjugate, that is $\nablal^{v}=-\dag\circ\nablar^{v}\circ\dag$. 
A pair of conjugate connections on $\X$ are both Hermitian if and only if for any right frame $(x_j)$ \cite[Proposition 2.30]{MRLC}
\begin{equation}
\overrightarrow{\nabla}(x_j)\ox x_j^\dag+x_j\ox\overleftarrow{\nabla}(x^\dag_j)=0.
\label{eq:conj-pair-Herm}
\end{equation}
\end{example}

The differential \eqref{eq: second-order-diff} allows us to  ask whether a connection on $\Omega^1_\D(\B)$ is torsion-free, meaning \cite[Section 4.1]{MRLC} that for any frame
\[
1\ox(1-\Psi)\big(\overrightarrow{\nabla}(\omega_j)\ox\omega_j^\dag+\omega_j\ox\d(\omega_j^\dag)\big)=0.
\]
For a Hermitian right connection, being torsion-free is equivalent to 
$(1-\Psi)\circ\overrightarrow{\nabla}=-\d$. For the conjugate left connection this becomes $(1-\Psi)\circ\overleftarrow{\nabla}=\d$, \cite[Proposition 4.5]{MRLC}.

Given a right frame $(\omega_j)\subset\Omega^1_\D(\B)$ we define
\[
W:=\d(\omega_j)\ox\omega_j^\dag\quad\mbox{and}\quad 
W^\dag:=\omega_j\ox\d(\omega_j^\dag).
\]

\begin{defn}
\label{ass:concordance}
Let $(\Omega^{1}_{\D}(\B),\dag,\Psi,\pairing{\cdot}{\cdot})$ be an Hermitian differential structure.
Define the projections $P:=\Psi\ox1$ and $Q:=1\ox\Psi$ on $T^{3}_{\D}(\B)$. The differential structure is concordant if
$T^{3}_{\D}=({\rm Im}(P)\cap{\rm Im}(Q))\oplus (\rm Im (1-P) + \rm Im (1-Q))$. Let $\Pi$ be the projection onto ${\rm Im}(P)\cap{\rm Im}(Q)$.
The differential structure is $\dag$-concordant if \cite[Definition 4.30]{MRLC}
\begin{equation}
(1+\Pi-PQ)^{-1}(W+PW^\dag)=(1+\Pi-QP)^{-1}(W^\dag+QW).
\label{eq:Bram's-condition}
\end{equation}
\end{defn}
The condition \eqref{eq:Bram's-condition} expresses a compatibility between $\Psi$, $\dag$, and the inner product, as encoded by the frame $(\omega_j)$. Importantly, despite being defined in terms of a frame, the three tensor
\[(1+\Pi-PQ)^{-1}(W+PW^\dag)-(1+\Pi-QP)^{-1}(W^\dag+QW),\]
is independent of the choice of frame. In particular, the
 $\dag$-concordance condition, which requires this three tensor to vanish, is frame independent \cite[Proposition 4.33]{MRLC}.

\begin{thm}
\label{thm:existence}
Let $(\Omega^{1}_{\D}(\B),\dag,\Psi,\pairing{\cdot}{\cdot})$ be an Hermitian differential structure. Then there exists an Hermitian and torsion-free (right) connection 
\[
\overrightarrow{\nabla}:\Omega^1_\D(\B)\to T^{2}_\D(\B)
\]
if and only if $(\Omega^{1}_{\D}(\B),\dag,\Psi,\pairing{\cdot}{\cdot})$ is  $\dag$-concordant.
\end{thm}
To obtain such a connection we use the maps $\alphar,\alphal$ of Equation \eqref{eq:alphas}, and add to the Grassmann connection $\nablar^v$ of a frame $v=(\omega_j)$ 
the one-form-valued endomorphism $\alphar(A)\in \overrightarrow{\textnormal{Hom}}^{*}_{\B}(\Omega^{1}_{\D}, T^{2}_{\D}),$  where
\[
A=-(1+\Pi-PQ)^{-1}(W+PW^\dag)\in T^{3}_{\D}.
\]
If instead we start with the left Grassmann connection $\nablal^v$ we subtract the connection form $\alphal(A)$. The two connections are conjugate.
\begin{example}
The construction for compact Riemannian manifolds yields the Levi-Civita connection on the cotangent bundle, \cite[Theorem 6.15]{MRLC}.
\end{example}

\subsection{Uniqueness of Hermitian torsion-free connections}
For uniqueness, we need the left and right representations $\alphal,\alphar$ as well as the definition of a special kind of bimodule connection.
\begin{defn}
\label{ass:ess}
Suppose that $\sigma: T^{2}_\D(\B)\to T^2_\D(\B)$ is an invertible bimodule map such that $\dag\circ\sigma=\sigma^{-1}\circ\dag$ and such that
the conjugate connections 
$\nablar,\nablal$ satisfy
\[
\sigma\circ\overrightarrow{\nabla}=\overleftarrow{\nabla}.
\]
Then we say that $\sigma$ is a braiding and that  $(\nablar,\sigma)$ is a $\dag$-bimodule connection.
\end{defn}
We denote by $\mathcal{Z}(M)$ the centre of a $\B$-bimodule $M$.
\begin{thm}
\label{thm:uniqueness}
Let  $(\Omega^{1}_{\D}(\B),\dag,\Psi,\pairing{\cdot}{\cdot})$ be a Hermitian differential structure. 
Suppose that $\sigma:T^{2}_{\D}\to T^{2}_{\D}$ is a braiding for which the map 
\[
\alphar+\sigma^{-1}\circ\alphal:\mathcal{Z}({\rm Im}(\Pi))\to\overleftrightarrow{{\rm Hom}}(\Omega^1_\D,T^2_\D)
\]
is injective. If there exists a Hermitian torsion-free $\sigma$-$\dag$-bimodule connection, then it is unique. 
\end{thm}

Even when we have the uniqueness given by Theorem \ref{thm:uniqueness} we do not have a closed formula for the part of the connection in ${\rm Im}(\Pi)$, but in examples this can usually be determined. For Riemannian manifolds, and indeed all examples so far, this part of the connection is zero.

\begin{defn}
If the $\dag$-concordant Hermitian differential structure $(\Omega^{1}_{\D}(\B),\dag,\Psi,\pairing{\cdot}{\cdot})$ admits a braiding $\sigma$ for which there exists a Hermitian torsion-free $\sigma$-$\dag$ bimodule connection, we call it the Levi-Civita connection, and denote it by $(\nablar^G,\sigma)$.
\end{defn}

For a compact Riemannian manifold $(M,g)$ equipped with a Dirac bundle $\slashed{S}\to M$, we have an associated spectral triple $(C^\infty(M),L^2(M,\slashed{S}),\slashed{D})$. Then $\Omega^{1}_{\slashed{D}}(C^\infty(M))\cong \Omega^1(M)\ox\C$ \cite[Chapter VI]{BRB}, and we let $\pairing{\cdot}{\cdot}_{g}$ be the inner product on $\Omega^{1}_{\slashed{D}}(C^\infty(M))$ induced by $g$. Moreover, the junk two-tensors in $T^{2}_{\slashed{D}}(C^{\infty}(M))\simeq T^{2}(M)$ coincide with the module of symmetric tensors \cite[Example 4.26]{MRLC}. Thus for $\sigma: T^{2}_{\slashed{D}}(M)\to T^{2}_{\slashed{D}}(M)$ the standard flip map we can set $\Psi:=\frac{1+\sigma}{2}$ for the junk projection. 

\begin{thm} 
\label{thm:LC-class}
Let $(M,g)$ be a compact Riemannian manifold with a Dirac bundle $\slashed{S}\to M$.  Then $(\Omega^{1}_{\slashed{D}}(C^{\infty}(M)),\dag, \Psi,\pairing{\cdot}{\cdot}_{g})$ is a $\dag$-concordant Hermitian differential structure and there exists a unique Hermitian torsion-free $\dag$-bimodule connection $(\nablar^{G},\sigma)$ on $\Omega^{1}_{\slashed{D}}(C^\infty(M))\cong \Omega^1(M)\ox\C$. The restriction $\nablar^{G}:\Omega^{1}(M)\ox1_\C\to \Omega^{1}(M)\ox1_\C$ coincides with the Riemannian connection on $\Omega^{1}(M)$.
\end{thm}
Other examples to which this machinary applies are $\theta$-deformations of Riemannian manifolds \cite[Section 6]{MRLC}, as well as pseudo-Riemannian manifolds,  \cite[Examples 2.4--2.6]{MRLC}, and the standard Podle\'{s} sphere \cite{MRPods}.




\section{Curvature}
\label{sec:curv}

The definitions of curvature we use are the classical ones, and have been used in the algebraic context for decades. The books \cite{Landi} and  \cite{BMBook} serve as excellent sources for the background, examples and related topics.

\subsection{Curvature for module connections and spectral triples}
A second order differential structure $(\Omega^{1}_{\D}(\B),\dag,\Psi)$ has a second order exterior derivative $\d:\Omega^{1}_{\D}(\B)\to \Lambda^{2}_{\D}(\B)$ which allows us to use the usual algebraic definition of curvature for a module connection.
\begin{defn}
Let $(\Omega^{1}_{\D}(\B),\dag,\Psi)$ be a second order differential structure, $\X_\B$ a finite projective right $\B$-module and $\nablar^{\X}:\X\to \X\otimes_{\B}\Omega^{1}_{\D}(\B)$ a connection. The curvature of $\X$ is the map $R^{\nablar^\X}:\X\to \X\ox_\B\Lambda^2_\D(\B)$ defined by
\[
R^{\nablar^\X}(x)=1\ox(1-\Psi)\circ (\nablar^\X\ox 1+1\ox\d)\circ\nablar^\X(x)\in \X\ox_\B\Lambda^2_\D(\B),\qquad x\in\X.
\]
Similarly, for a connection $\nablal^\X$ on a left module ${}_\B\X$ we define
the curvature to be
\[
R^{\nablal^\X}(x)=(1-\Psi)\ox1\circ (1\ox\nablal^\X-\d\ox1)\circ\nablal^\X(x)\in \Lambda^2_\D(\B)\ox_\B\X,\qquad x\in\X.
\]
\end{defn}

The sign difference between the left and right curvatures is due to the fact that $\d$ satisfies a graded Leibniz rule while connections do not interact with the grading in such a way. For a pair of conjugate connections $\nablal=-\dag\circ\nablar\circ\dag$
on a $\dag$-bimodule $\X$ the curvatures are related via
\[
R^{\nablar}(x)^\dag=R^{\nablal}(x^\dag).
\]

The next lemma provides tools for computing curvature.

\begin{lemma}
\label{lem:frame-diff}
Let $(\Omega^{1}_{\D}(\B),\dag,\Psi,\pairing{\cdot}{\cdot})$ be an Hermitian differential structure,  $\X_\B$ a finite projective right $\B$-module
 and $v=(x_j)$ a  frame for $\X_\B$.
Any right connection $\nablar^\X$ can be written $\nablar^\X=\nablar^v+\alphar(A)$ where $\nablar^v$ is the Grassmann connection and $A\in \X\ox\Omega^1_\D\ox\X$ is given by
$A=\sum_j\nablar(x_j)\ox x_j^\dag$. Writing $A=\sum_{j,k}x_j\ox A^k_j\ox x_k^\dag$ we have
\[
\sum_k\pairing{x_j}{x_k}_\B\,A_k^\ell=A^\ell_j,\qquad \sum_kA^k_j\pairing{x_k}{x_\ell}_\B=A^\ell_j,\qquad \sum_{j,k}x_j\ox[\D,\pairing{x_j}{x_k}_\B]\ox x_k^\dag=0.
\]
\end{lemma}
\begin{proof}
The first two statements come from the frame relation
\[
\sum_{j,k}x_j\ox A^k_j\ox x_k^\dag=\sum_{j,k,p}x_p\ox\pairing{x_p}{x_j}_\B A_j^k\ox x_k^\dag=\sum_{j,k,p}x_j\ox A_j^k\pairing{x_k}{x_p}_\B\ox x_p^\dag.
\]
The third is similar, with
\begin{align*}
\sum_j\nablar(x_j)\pairing{x_j}{x}_\B&=\sum_{j,k}\nablar(x_k\pairing{x_k}{x_j}_\B)\pairing{x_j}{x}_\B\\
&=\sum_k\nablar(x_k)\pairing{x_k}{x}_\B+\sum_{j,k}x_k\ox[\D,\pairing{x_k}{x_j}_\B]\pairing{x_j}{x}_\B.\qedhere
\end{align*}
\end{proof}
What follows is essentially the classical calculation showing that $R^{\nablar^\X}$ is a two-form-valued endomorphism, and can be found in \cite{Landi,MRS}. 
\begin{prop}
\label{prop:module-curve}
Let $(\Omega^{1}_{\D}(\B),\dag,\Psi,\pairing{\cdot}{\cdot})$ be an Hermitian differential structure, $\X_{\B}$ a finitely generated projective module, $\nablar^\X$ a right connection. The curvature is a well-defined two-form-valued endomorphism $R^{\nablar^{\X}}\in\Hom^{*}_{\B}(\X,\X\otimes\Lambda^{2}_{\D})$. If $v=(x_j)$ is a frame for $\X$ and $\nablar^{X}=\nablar^{v}+\alphar(A)$ then
\begin{align*}
R^{\nablar^\X}(x)=1\ox(1-\Psi)&\Big(x_k\ox[\D,\pairing{x_k}{x_j}_\B]\ox[\D,\pairing{x_j}{x_p}_\B]\pairing{x_p}{x}_\B\\
& +x_l\ox A^k_l\ox A^j_k\pairing{x_j}{x}_\B+x_k\ox\d(A^j_k)\pairing{x_j}{x}_\B)\Big)
\end{align*}
Similarly the curvature of a left connection $\nablal^{\X}=\nablal^{v}+\alphal(A)$ is a well-defined two-form-valued endomorphism $R^{\nablal^{X}}\in \Hom^{*}_{\B}(\X,\Lambda^{2}_{\D}\otimes\X)$, and
\begin{align*}
R^{\nablal^\X}(x)=(1-\Psi)\ox1&\Big({}_\B\pairing{x}{x_p}[\D,{}_\B\pairing{x_p}{x_j}]\ox[\D,{}_\B\pairing{x_j}{x_k}]\ox x_k\\
&+{}_\B\pairing{x}{x_j}A_j^l\ox A_l^k\ox x_k-{}_\B\pairing{x}{x_j}\d(A_j^k)\ox x_k\Big).
\end{align*}
\end{prop}
\begin{proof}
We prove the result for right modules.
Fixing a frame $(x_j)$ of $\X_\B$, write
\[
\nablar^\X(x)=x_j\ox[\D,\pairing{x_j}{x}_\B]+x_j\ox A^k_j\pairing{x_k}{x}_\B,\qquad x\in\X_\B.
\]
Given $x\in \X_\B$ we use Lemma \ref{lem:frame-diff} repeatedly to find  
\begin{align*}
R^{\nablar^\X}&(x)=1\ox(1-\Psi)\circ \Big(\nablar^\X\ox 1+1\ox\d)(x_j\ox[\D,\pairing{x_j}{x}_\B]+x_k\ox A^j_k\pairing{x_j}{x}_\B\Big)\\
&=1\ox(1-\Psi)\Big(\nablar^\X(x_j)\ox[\D,\pairing{x_j}{x}_\B]+\nablar^\X(x_k)\ox A^j_k\pairing{x_j}{x}_\B+x_k\ox\d(A^j_k\pairing{x_j}{x}_\B)\Big)\\
&=1\ox(1-\Psi)\Big(x_k\ox[\D,\pairing{x_k}{x_j}_\B]\ox[\D,\pairing{x_j}{x}_\B]+x_k\ox A^l_k\pairing{x_l}{x_j}_\B\ox[\D,\pairing{x_j}{x}_\B]\\
&\qquad+x_l\ox[\D,\pairing{x_l}{x_k}_\B]\ox A^j_k\pairing{x_j}{x}_\B+x_l\ox A^m_l\pairing{x_m}{x_k}_\B\ox A^j_k\pairing{x_j}{x}_\B\\
&\qquad+x_k\ox\d(A^j_k)\pairing{x_j}{x}_\B)-x_k\ox A^j_k\ox[\D,\pairing{x_j}{x}_\B]\Big)\\
&=1\ox(1-\Psi)\Big(x_k\ox[\D,\pairing{x_k}{x_j}_\B]\ox[\D,\pairing{x_j}{x_p}_\B]\pairing{x_p}{x}_\B\\
&\qquad+x_l\ox[\D,\pairing{x_l}{x_k}_\B]\ox \pairing{x_k}{x_m}_\B A^j_m\pairing{x_j}{x}_\B+x_l\ox A^m_l\pairing{x_m}{x_k}_\B\ox A^j_k\pairing{x_j}{x}_\B\\
&\qquad+x_k\ox\d(A^j_k)\pairing{x_j}{x}_\B)\Big)\\
&=1\ox(1-\Psi)\Big(x_k\ox[\D,\pairing{x_k}{x_j}_\B]\ox[\D,\pairing{x_j}{x_p}_\B]\pairing{x_p}{x}_\B\\
&\qquad +x_l\ox A^k_l\ox A^j_k\pairing{x_j}{x}_\B+x_k\ox\d(A^j_k)\pairing{x_j}{x}_\B)\Big).
\end{align*}
The case of a left connection follows similarly.\end{proof}

The advantage of using a global frame for computing curvature, even classically, is that the topological contribution to the curvature is separated out in the Grassmann term 
\[
1\ox(1-\Psi)\Big(x_k\ox[\D,\pairing{x_k}{x_j}_\B]\ox[\D,\pairing{x_j}{x_p}_\B]\Big)\pairing{x_p}{x}_\B
\]
with the connection form  contributions ``$\d A+A\wedge A$'' being purely geometric.

For a compact Riemannian manifold $(M,g)$ equipped with a Dirac bundle $\slashed{S}\to M$ and associated spectral triple $(C^{\infty}(M), L^{2}(M,\slashed{S}),\slashed{D})$ we have $\Omega^{1}_{\slashed{D}}(C^{\infty}(M))\simeq \Omega^{1}(M)\otimes\mathbb{C}$ and $\Psi\!=$\,symmetrisation projection. It is well-known that this notion of connection coincides with the usual one in the case of a connection on a smooth Riemannian vector bundle $E\to M$. Consequently, the definition of curvature applied to such a connection also recovers the usual geometric curvature tensor. 

In view of Theorem \ref{thm:LC-class}, we can recover the Riemann tensor of the manifold $M$ by considering the curvature of the Levi-Civita connection. This motivates the following definition.
\begin{defn} Let $(\B,\H,\D)$ be a spectral triple admitting a $\dag$-concordant Hermitian differential structure $(\Omega^{1}_{\D}(\B), \dag,\Psi,\pairing{\cdot}{\cdot})$ and a braiding $\sigma:T^{2}_{\D}\to T^{2}_{\D}$ for which there exists a Hermitian torsion-free $\dag$ bimodule connection $(\nablar^{G},\sigma)$. The curvature tensor of $(\B,\H,\D)$ is then defined to be $R^{\nablar^{G}}$.
\end{defn}

\subsection{Ricci and scalar curvature}

The Ricci and scalar curvature are obtained from the full Riemann tensor by taking  traces in suitable pairs of variables. The analogue in our setting is the inner product with the line element 
$G\in T^{2}_{\D}(\B)$ introduced in Equation \eqref{eq:Gee}, again for suitable pairs of variables. Similar definitions appear in \cite[p574ff]{BMBook}.

For a manifold, we can choose a frame coming from orthonormalising local coordinates
\begin{equation}
\omega^{k}_\alpha=\sqrt{\varphi_\alpha}B^k_\mu\,\dee x^\mu_\alpha
\label{eq:mfld-coords}
\end{equation}
with the the help of a partition of unity $\varphi_\alpha$. 
Here we abuse notation by writing $\dee x^\mu$ for $[\D,x^\mu_\alpha]$ computed locally, $B^k_\mu \dee x^\mu_\alpha=e^k_\alpha$ is a local orthonormal frame,
and where for a self-adjoint or symmetric operator $\D$ we have
$(\dee x^\mu)^\dag=-\dee x^\mu$. 

\begin{prop}
\label{lem:ricci}
Let $(M,g)$ be a Riemannian manifold and set $\B=C^\infty(M)$.
Identifying tangent and cotangent bundles of the Riemannian manifold $(M,g)$, the curvature tensor is (locally) 
\[
R=\dee x^\mu \ox R_{\sigma\rho\mu}^{\ \ \ \ \nu}\dee x^\sigma\wedge \dee x^\rho\ox (\dee x^\nu)^\dag.
\]
The Ricci tensor is 
\[
{\rm Ric}={}_\B\pairing{R}{G}
\]
and the scalar curvature is
\[
r=\pairing{G}{{\rm Ric}}_\B.
\]
\end{prop}
\begin{proof}
We will work over a single chart.
Writing $\dee x^\sigma\wedge \dee x^\rho$ as $\frac{1}{2}(\dee x^\sigma\ox \dee x^\rho-\dee x^\rho\ox \dee x^\sigma)$ allows us to
compute the left inner product with the identity operator $G$. Locally $G=g_{\alpha\beta}\dee x^\alpha\ox(\dee x^\beta)^\dag$, 
so we find
\begin{align*}
{}_{\B}\pairing{R}{G}&=\frac{1}{2}\dee x^\mu \ox R_{\sigma\rho\mu}^{\ \ \ \ \ \nu}\dee x^\sigma {}_{C(M)}\pairing{\dee x^\rho\ox (\dee x^\nu)^\dag}{g_{\alpha\beta}\dee x^\alpha\ox (\dee x^\beta)^\dag}\\
&-\frac{1}{2}\dee x^\mu \ox R_{\sigma\rho\mu}^{\ \ \ \ \ \nu}\dee x^\rho {}_{C(M)}\pairing{\dee x^\sigma\ox (\dee x^\nu)^\dag}{g_{\alpha\beta}\dee x^\alpha\ox (\dee x^\beta)^\dag}\\
&=-\frac{1}{2}\dee x^\mu \ox R_{\sigma\rho\mu\nu}\dee x^\sigma g^{\beta\nu}g_{\alpha\beta}g^{\rho\alpha}
+\frac{1}{2}\dee x^\mu \ox R_{\sigma\rho\mu\nu}\dee x^\rho g_{\alpha\beta}g^{\nu\beta}g^{\sigma\alpha}\\
&=-\frac{1}{2}\dee x^\mu \ox R_{\sigma\rho\mu\nu}\dee x^\sigma g^{\rho\nu}
+\frac{1}{2}\dee x^\mu \ox R_{\sigma\rho\mu\nu}\dee x^\rho g^{\sigma\nu}\\
&=-\dee x^\mu \ox \dee x^\sigma R_{\sigma\rho\mu\nu} g^{\rho\nu}
=-\dee x^\mu \ox \dee x^\sigma R_{\sigma\rho\mu\nu}
=\dee x^\mu \ox (\dee x^\sigma)^\dag R_{\sigma\rho\mu}^{\ \ \ \ \rho}. 
\end{align*}
The scalar curvature is then the right inner product of the Ricci curvature with the identity operator,
\begin{align*}
\pairing{G}{{}_{\B}\pairing{R}{G}}_{\B}
&=\pairing{g_{\alpha\beta}\dee x^\alpha\ox (\dee x^\beta)^\dag }{\dee x^\mu \ox (\dee x^\sigma)^\dag}_{\B}R_{\sigma\rho\mu\nu} g^{\rho\nu}\\
&=R_{\sigma\rho\mu\nu} g^{\rho\nu}g_{\alpha\beta}g^{\sigma\beta}g^{\alpha\mu}\\
&=R_{\sigma\rho\mu\nu}g^{\rho\nu}g^{\sigma\mu}\\
&=R_{\sigma\rho\mu\nu}g^{\rho\nu}g^{\sigma\mu}.\qedhere
\end{align*}
\end{proof}

As a consequence of these computations, we see that we can
define the Ricci and scalar curvature for any connection on the one-forms of an Hermitian differential structure.

\begin{defn}
Let $(\Omega^{1}_{\D}(\B),\dag,\Psi,\pairing{\cdot}{\cdot})$ be a Hermitian differential structure and $\nablar$ a right connection on $\Omega^{1}_{\D}$ with curvature $R^{\nablar}\in \Omega^1_\D(\B)\ox\Lambda^2_\D(\B)\ox\Omega^1_\D(\B)$.
The Ricci curvature of $\nablar$ is
\[
{\rm Ric}^{\nablar}={}_\B\pairing{R^{\nablar}}{G}\in T^2_\D(\B)
\]
and the scalar curvature is
\[
r^{\nablar}=\pairing{G}{{\rm Ric}^{\nablar}}_\B.
\]
\end{defn}
These definitions mirror those of \cite{BMBook} and references therein, and agree when both apply.
We will compute these curvatures for $\theta$-deformations of compact manifolds in Section \ref{subsec:no-change}, and in \cite{MRPods} we will examine the curvature of the Podle\'{s} sphere.


\section{Weitzenbock formula}
\label{sec:Weitzenbock}

In this section we relate the second covariant derivative to the connection Laplacian. Under additional assumptions, mimicking the definition of Dirac bundles on manifolds, we compare the connection Laplacian to $\D^2$. 

Before introducing our definition of Dirac spectral triples, modelled on the definition for manifolds, we clarify the role of the flip map which is present even in the commutative case. These observations influence the general form of Weitzenbock formulae even for manifolds.

\subsection{Clifford connections and braiding}
\label{subsec:in-detail}

The definition of Clifford connection and Dirac bundle for Riemannian manifolds
as found in \cite[Chapter II, Section 5]{LM} tacitly makes use of 
commutativity in a number of ways. Here we clarify where commutativity is used and provide motivation for the appearance of the braiding in the definition of Dirac spectral triple (Definition \ref{def: Dirac-spectral-triple}).

In the setting of a Dirac bundle $\slashed{S}\to M$ on a Riemannian manifold $(M,g)$, we write $\X:=\Gamma(M,\slashed{S})$ for the central bimodule of sections of $\slashed{S}$. One of the requirements of a Dirac bundle is that module of one-forms $\Omega^{1}(M)$ acts as endomorphisms of $\X$. That is, we are given a $C^{\infty}(M)$-linear map $\Omega^{1}(M)\to \End(\X)$.

We say that a connection $\nabla^\X$ is a Clifford connection if
given a one form $\omega$, the Levi-Civita connection $\nabla^g$, and a section $x\in \X$, we have
\begin{equation}
\omega\ox x\mapsto \nabla^\X(\omega\cdot x)=\nabla^g(\omega)\cdot x+\omega\cdot\nabla^\X(x).
\label{eq:classy-Cliff}
\end{equation}
In order for the right hand side of \eqref{eq:classy-Cliff} to be well-defined 
on the balanced tensor product $\Omega^1\ox_{C^\infty(M)}\X$ (before letting the one-form part act) 
requires $\nabla^{g}$ to be a right connection and $\nabla^{\X}$ to be a left connection. Of course in the commutative case, any right connection can be turned into a left connection using the flip map $ \X\ox_{C^\infty(M)}\Omega^{1}\to \Omega^{1}\ox_{C^\infty(M)} \X$, $x\ox\omega\mapsto \omega\ox x$. 
 
Ensuring well-definedness forces us to work with a left and a right connection, but subsequently, care is required when properly defining the action of the endomorphism defined by the one form $\omega$ on $\Omega^{1}\ox_{C^{\infty}(M)} \X$. Since $C^{\infty}(M)$ commutes with $\End(\X)$ the operator $(1\ox\omega)(\eta\ox x):=\eta\ox\omega\cdot x$ is well-defined on $\Omega^{1}\ox_{C^{\infty}(M)} \X$. In the noncommutative setting, this is no longer true.

The issue can be overcome by using a braiding $\sigma:\Omega^{1}\ox_{C^\infty(M)}\Omega^{1}\to \Omega^{1}\ox_{C^\infty(M)}\Omega^{1}$, which in the commutative case would be the flip map. In that case
\[
(\sigma(\omega\ox\eta))\cdot x = \eta\ox\omega\cdot x=(1\ox\omega)(\eta\ox x),
\]
and in this equation the left hand side can be generalised by using a braiding, whereas the right hand side does not generally make sense. The classical Clifford connection condition can thus be rewritten 
in terms of left and right connections as
\begin{equation}
\nablal^{\X}(c(\omega)x)=(1\ox c)(\sigma\ox 1)(\omega\ox\nablal^{\X}(x)+\nablar^{G}(\omega)\ox x),
\label{eq:Cliff-flip}
\end{equation}
and in this form can be reinterpreted in the noncommutative context.

\subsection{Dirac spectral triples and the connection Laplacian}

We now introduce a class of spectral triples for which the Weitzenbock formula holds.  Given a left inner product module $\X$ and a positive functional $\phi:\B\to\C$, the Hilbert space $L^2(\X,\phi)$ is the completion of
$\X$ in the scalar product $\langle x,y\rangle:=\phi({}_\B\pairing{x}{y})$.
\begin{defn} 
\label{def: Dirac-spectral-triple}
Let $(\B,\H,\D)$ be a spectral triple equipped with a braided Hermitian differential structure $(\Omega^{1}_{\D}(\B),\dag,\Psi,\pairing{\cdot}{\cdot},\sigma)$. Then $(\B,\H,\D)$ is a \emph{Dirac spectral triple} over $(\Omega^{1}_{\D}(\B),\dag,\Psi,\pairing{\cdot}{\cdot},\sigma)$ 
if 
\begin{enumerate}
\item for $\omega,\eta\in\Omega^{1}_{\D}(\B)$ we have
\begin{equation}
\label{eq:Clifford}
(m\circ\Psi)(\rho\ox\eta)=e^{-\beta}m(G)\pairing{\rho^\dag}{\eta}_\B=-e^{-\beta}m(G)g(\rho\ox\eta); 
\end{equation}
\item there is left inner product module $\X$ over $\B$ and a positive functional $\phi:\B\to\mathbb{C}$ such that $\H=L^{2}(\X,\phi)$ and the natural map $c:\Omega^{1}_{\D}(\B)\otimes_{\B}L^{2}(\X,\phi)\to L^{2}(\X,\phi)$ restricts to a map $c:\Omega^{1}_{\D}(\B)\otimes_{\B}\X\to \X$;
\item There is a left connection $\nablal^{\X}:\X\to \Omega^{1}_{\D}(\B)\otimes_{\B}\X$ such that $\D=c\circ\nablal^{\X}:\X\to L^{2}(\X,\phi)$;

\item there is a Hermitian torsion free $\dag$-bimodule connection $(\nablar^{G},\sigma)$ on $\Omega^{1}_{\D}$ such that
\begin{align}
\D(\omega x)\!=c\circ\nablal^\X(c(\omega\ox x))&\!=c\circ (m\circ\sigma\ox 1)(\nablar^{G}\!(\omega)\!\ox x)+\omega\ox\nablal^\X(x))
\label{eq:compatible2}
\end{align}
\end{enumerate}
\end{defn}
The well-known order one condition for spectral triples
gives a sufficient condition for $\D$ to be of the form $c\circ\nablal^\X$,
\cite[Section 3]{LRV}. 
The compatibility of $\nablal^\X$ with $\nablar^{G}$
is the analogue of the `Clifford connection' condition on a 
Dirac bundle, \cite[Definition 5.2]{LM}. Condition 1. captures the 
essential feature of Clifford multiplication, namely that the product
is that of differential forms modulo the line element $G$.

As discussed subsection \ref{subsec:in-detail}, the definition of Dirac bundle for manifolds exploits
commutativity to ensure the one-forms act in the correct order in condition 4. The bimodule map $\sigma$ plays the role of the flip map to do the same job in the noncommutative context. 
\begin{rmk}
\label{bimodremark}
One could consider the Clifford connection condition \eqref{eq:compatible2} relative to an arbitrary right connection $\nablar^{\Omega^{1}}$. Computing $[\D,a]\omega x$ then gives that 
\[
m(\sigma\nablar^{\Omega^{1}}(a\omega))=[\D,a]\omega +a m(\sigma\nablar^{\Omega^{1}}(\omega)).
\]
Hence $\nablar^{\Omega^{1}}$ is forced to be a $\sigma$-bimodule connection modulo $\ker m$.
\end{rmk}

We will consider examples of Dirac spectral triples, such as $\theta$-deformations of classical Dirac bundles in the present paper and the Podle\'{s} sphere in \cite{MRPods}. 

Consider the left $\B$-module $\X$. The curvature of $\X$ is given by the covariant second derivative $(1-\Psi)\ox1\circ(1\ox\nablal^\X-\d\ox1)\circ\nablal^\X$. The existence of the connection $\nablar^{G}$ on $\Omega^{1}_{\D}(\B)$ allows us define a connection Laplacian,
via a second derivative of the form $(1\ox\nablal^\X+\nablar^{G}\ox1)\circ\nablal^\X$, combined with the analogue of a trace map on two-tensors. By \cite[Proposition 2.30]{MRLC}
$$
 \nablar^{G}\ox1+1\ox\nablal^\X:\Omega^1_\D(\B)\ox_\B\X\to
T^{2}_{\D}(\B)\ox_\B\X
$$ 
is well-defined. In this section we will show how to construct the connection Laplacian for Dirac spectral triples.  
\begin{defn}
\label{defn:laplace}
Let $(\B,H,\D)$ be a Dirac spectral triple with braided Hermitian differential structure $(\Omega^{1}_{\D}(\B),\dag,\Psi,\pairing{\cdot}{\cdot},\sigma)$, and $\nablal^{\X}:\X\to \Omega^{1}_{\D}\otimes_{\B}\X$ and $\nablar^{G}:\Omega^{1}_{\D}\to T^{2}_{\D}(\B)$ the associated connections. 
With $m:T^2_\D(\B)\to \Omega^2_\D(\B)$ the multiplication map, we define the  
{\em connection Laplacian of $\nablal^{\X}$ relative to $\nablar^{G}$} by
$$
\Delta^\X(x):=e^{-\beta}m(G)\pairing{G}{(\nablar^{G}\ox 1+1\ox \nablal^{\X})\circ \nablal^\X(x)}_\X\in\X.
$$
\end{defn}

Note that $\pairing{G}{(\nablar^{G}\ox 1+1\ox \nablal^{\X})\circ \nablal^\X(x)}_\X\in\X$ since $(\nablar^{G}\ox 1+1\ox \nablal^{\X})\circ \nablal^\X(x)\in T^{2}_{\D}(\B)\ox_{\B}\X$. Moreover Condition 2 of Definition \ref{def: Dirac-spectral-triple} guarantees that $m(G)$ maps $\X$ to itself, so that indeed $\Delta^{\X}:\X\to\X$.

For commutative manifolds and Dirac-type operators,
this definition specialises to the usual connection
Laplacian when $\X$ is the module of smooth sections of a vector bundle and $\nablar^{G}$ is the Levi-Civita connection on the cotangent bundle. The 
operator $e^{-\beta}G\bra{G}$ is the projection onto the span of $G$ in $T^2_\D$, and the next Lemma describes $e^{-\beta}m(G)$ for manifolds.
\begin{lemma}
\label{lem:dim-emm}
For a compact Riemannian manifold $(M,g)$ equipped with a Dirac bundle $\slashed{S}\to M$ and associated spectral triple $(C^{\infty}(M),L^{2}(M,\slashed{S}),\slashed{D})$, the operator $e^{-\beta}m(G)$ is the identity, and so $\Delta^\X$ is the usual connection Laplacian.
\end{lemma}
\begin{proof}
On a Riemannian manifold, the line element is $G=\sum_{\alpha,\mu,\nu}\varphi_\alpha g_{\mu\nu}\gamma^\mu\ox\gamma^{\nu*}$. This is expressed using a covering of the manifold by charts $U_\alpha$ with partition of unity $\varphi_\alpha$ and coordinates $x^\mu$ whose differentials are represented by $\gamma({\rm d}x^\mu)=:\gamma^\mu$ (see Equation \eqref{eq:mfld-coords}). Computing locally (as we may) using the Clifford relations gives
\[
m(g_{\mu\nu}\gamma^\mu\ox\gamma^{\nu*})=g_{\mu\nu}g^{\mu\nu}{\rm Id}=\dim(M){\rm Id}
\]
and $e^{\beta}=\dim(M)$. The formula for $\Delta^\X$ reduces to the classical formula for the connection Laplacian, and so we are done.
\end{proof}

\begin{rmk}
The choice of inner product on $\Omega^1_{\slashed{D}}$ is critical to
Lemma \ref{lem:dim-emm}. The Clifford elements $\gamma^\mu$ encode the metric $g$ used to define $\slashed{D}$, and if we take a different Riemannian metric $h$ on $\Omega^1_{\slashed{D}}$ we find $m(G)=h_{\mu\nu}g^{\mu\nu}{\rm Id}$.
\end{rmk}

\subsection{The Weitzenbock formula for Dirac spectral triples}

Given a Dirac spectral triple $(\B,L^2(\X,\phi),\D)$,
we can compare the action of $\D^2$ on $L^2(\X,\phi)$ with 
the connection Laplacian.

\begin{thm}
\label{thm:Weitz1} 
Let $(\B,L^2(\X,\phi),\D)$ be a Dirac spectral triple relative to $(\Omega^{1}_{\D}(\B),\dag,\Psi,\pairing{\cdot}{\cdot},\sigma)$ and $\Delta^{\X}$ the connection Laplacian of the left connection $\nablal^\X$.
If $m\circ\sigma\circ\Psi=m\circ\Psi$ and $\Psi(G)=G$ then
\begin{align}
\label{eq:Weitzy}
\D^2(x)
&=\Delta^\X(x)+c\circ (m\circ\sigma\ox 1)\big(R^{\nablal^\X}(x)\big).
\end{align}

\end{thm}
\begin{proof}
We start our comparison of $\D^2$ and $\Delta^\X$ by 
using Equation \eqref{eq:compatible2} to write 
\begin{align}
\nonumber\D^2(x)&
=c\circ (m\circ\sigma\ox 1)(\nablar^{G}\ox 1+1\ox\nablal^{\X})\circ\nablal^\X(x)))\\
\label{eq:laplaceterm}&=c\circ (m\circ\sigma\ox 1)\Big((\Psi\ox1)(\nablar^{G}\ox 1+1\ox\nablal^{\X})\circ\nablal^\X(x)\Big)\\
\label{eq: htfweitzy}&\qquad+c\circ (m\circ\sigma\ox 1)\Big(R^{\nablal^\X}(x)+(1-\Psi)(\nablar^{G}+\d)\ox1\circ\nablal^\X(x))\Big)\\
\nonumber&=c\circ (m\circ\sigma\ox 1)\Big((\Psi\ox1)(\nablar^{G}\ox 1+1\ox\nablal^{\X})\circ\nablal^\X(x)\Big)+c\circ (m\circ\sigma\ox 1)\Big(R^{\nablal^\X}(x)\Big).
\end{align}
We will identify the term \eqref{eq:laplaceterm} with the connection Laplacian $\Delta^{\X}$. By \eqref{eq:Clifford} we have
\[
m(\omega\ox\rho)=e^{-\beta}m(G)\pairing{\omega^\dag}{\rho}_\B+m(1-\Psi)(\omega\ox\rho).
\]
So if $\Psi(\omega\ox\rho)=\omega\ox\rho$ then
\[
c\circ (m\ox 1)(\omega\ox\rho\ox x)=e^{-\beta}m(G)\pairing{\omega^\dag}{\rho}_\B x=e^{-\beta}m(G)\pairing{G}{\omega\ox\rho}_\B x.
\]
Since $m\circ\sigma\circ\Psi=m\circ\Psi$ and $\Psi(G)=G$ we have  
\begin{align*}
c\circ((m\circ\sigma\circ\Psi)\ox1)&(\nablar^{G}\ox 1+1\ox\nablal^{\X})\circ\nablal^\X(x))\\&=c\circ((m\circ\Psi)\ox1)(\nablar^{G}\ox 1+1\ox\nablal^{\X})\circ\nablal^\X(x))\\
&=e^{-\beta}m(G)\pairing{G}{(\Psi\ox 1)(\nablar^{G}\ox 1+1\ox\nablal^{\X})\circ\nablal^\X(x)}\\
&=e^{-\beta}m(G)\pairing{G}{(\nablar^{G}\ox 1+1\ox\nablal^{\X})\circ\nablal^\X(x)}=\Delta^\X(x)
\end{align*}
which completes the proof.
\end{proof}
\begin{rmk}
\label{htfweitzyremark}
As in Remark \ref{bimodremark}, one could consider the Clifford connection condition \eqref{eq:compatible2} relative to an arbitrary right connection $\nablar^{\Omega^{1}}$. Equation \eqref{eq: htfweitzy} can then be derived and we see that in order for the Weitzenbock formula to hold, the connection $\nablar^{\Omega^{1}}$ is forced to be Hermitian and torsion free modulo $\ker m$.
\end{rmk}
\subsection{Divergence condition for positivity of the Laplacian}
In geometric calculations, the fact that the connection Laplacian is a positive Hilbert space operator is essential. In this subsection we derive an abstract condition guaranteeing positivity, corresponding to the well-known fact that the integral of the divergence of a vector field vanishes. Although we will not use this condition in the present paper, we record it for completeness.
We first observe the following.

\begin{lemma}
\label{lem:adjoint-conn}
Suppose that $(\Omega^{1}_{\D}(\B),\dag,\Psi,\pairing{\cdot}{\cdot})$ a Hermitian differential structure and $\X$ a left $\B$-inner product module.
For $\nablal^{\X}:\X\to \Omega^1_\D(\B)\ox_\B \X$ an Hermitian left connection
and $\nablar^{\Omega^{1}}:\Omega^1_\D(\B)\to T^{2}_{\D}(\B)$ a right connection, we have
\[
\pairing{G}{{}_\B\pairing{\nablal^\X x}{\nablal^\X y}}_\B={}_\B\pairing{{}_{T^2}\pairing{(\nablar^{\Omega^{1}}\ox1+1\ox\nablal^\X)\circ\nablal^\X(x)}{y}-\nablar^{\Omega^{1}}({}_{\Omega^1}\pairing{\nablal^\X x}{y})}{G}.
\]
\end{lemma}
\begin{proof}
For $x\in \X$ we write $\nablal^\X(x)=\omega_{(0)}\ox x_{(1)}$ as a Sweedler sum. Then we use the Leibniz rule for $\nablar^\Omega$ and Hermitian property for $\nablal^\X$ to obtain
\begin{align*}
&{}_\B\pairing{{}_{T^2}\pairing{(\nablar^{\Omega^1}\ox1+1\ox\nablal^\X)(\omega_{(0)}\ox x_{(1)})}{y}}{G}\\
&\qquad={}_\B\pairing{\nablar^{\Omega^{1}}(\omega_{(0)}){}_\B\pairing{x_{(1)}}{y}}{G}+{}_\B\pairing{\omega_{(0)}\ox{}_{\Omega^1}\pairing{\nablal^\X x_{(1)}}{y}}{G}\\
&\qquad={}_\B\pairing{\nablar^{\Omega^{1}}(\omega_{(0)}{}_\B\pairing{x_{(1)}}{y})}{G}-{}_\B\pairing{\omega_{(0)}\ox[\D,{}_\B\pairing{x_{(1)}}{y}]}{G}\\
&\qquad\qquad+{}_\B\pairing{\omega_{(0)}\ox{}_\B\pairing{x_{(1)}}{\nablal^\X y}}{G}+{}_\B\pairing{\omega_{(0)}\ox[\D,{}_\B\pairing{x_{(1)}}{y}]}{G}\\
&\qquad={}_\B\pairing{\nablar^{\Omega^{1}}(\omega_{(0)}{}_\B\pairing{x_{(1)}}{y})}{G}+{}_\B\pairing{{}_{T^2}\pairing{\nablal^\X x}{\nablal^\X y}}{G}.
\end{align*}
The statement now follows by observing that for any two-tensor $\rho\ox \eta\in \Omega^1_\D(\B)^{\ox2}$ we have
\[
{}_\B\pairing{\rho\ox\eta}{G}={}_\B\pairing{\rho}{\eta^\dag}=\pairing{\rho^\dag}{\eta}_\B=\pairing{G}{\rho\ox\eta}_\B.
\qedhere
\]
\end{proof}
For a Dirac spectral triple $(\B,L^{2}(\X,\phi),\D)$, $\Omega^1_\D(\B)\ox_\B\H\cong L^2(\Omega^1_\D(\B)\ox_\B \X,\phi)$ with inner product
\[
\langle \rho\ox x,\eta\ox y\rangle
=\phi({}_\B\pairing{\rho\ox x}{\eta\ox y})
=\phi({}_\B\pairing{\rho\,{}_\B\pairing{x}{y}}{\eta})
=\phi({}_\B\pairing{\rho\,{}_\B\pairing{x}{y}\ox \eta^\dag}{G}).
\]

Recognising ${}_\B\pairing{\nablar^{\Omega^{1}}({}_\B\pairing{\nablal^\X x}{y}^\X)}{G}$ as a divergence term, the centrality of $e^{-\beta}m(G)$ gives us essentially the classical argument for the positivity of the connection Laplacian.

\begin{corl}
\label{cor:for-pos-curv}
Let $(\B,L^{2}(\X,\phi),\D)$ be a Dirac spectral triple over the braided Hermitian differential structure $(\Omega^{1}_{\D}(\B),\dag,\Psi,\pairing{\cdot}{\cdot},\sigma)$. If $\phi({}_\B\pairing{\nablar^G(\omega_{(0)}{}_\B\pairing{x_{(1)}}{x}^{\X})}{G})=0$ then 
\[
\phi((\langle \Delta^\X(x),x\rangle_\B)=\phi\big((e^{-\beta}m(G))^{1/2}\pairing{\nablal^{\X}(x)}{\nablal^{\X}(x)}_\B(e^{-\beta}m(G))^{1/2}\big)\geq 0.
\]
\end{corl}


\section{Connections and curvature for $\theta$-deformations}
\label{subsec:no-change}
\subsection{Background and notation}

Let $(M,g)$ be a compact Riemannian manifold equipped with a Dirac bundle $\slashed{S}\to M$ and $(C^\infty(M),L^2(M,\slashed{S}),\slashed{D})$ the associated Dirac spectral triple (in the sense of Definition \ref{def: Dirac-spectral-triple}). 

The space of $1$-forms $\Omega^{1}_{\slashed{D}}(M)\simeq \Omega^{1}\otimes \mathbb{C}$ acts via the Clifford action on $L^2(M,\slashed{S})$, and so carries a $\dag$-operation induced by operator adjoint $T\mapsto T^*$, as well as an inner product $\pairing{\cdot}{\cdot}_{g}$ induced by the Riemannian metric $g$. Moreover, the standard flip map $\sigma: T^{2}_{\slashed{D}}(M)\to T^{2}_{\slashed{D}}(M)$ gives the junk projection  $\Psi:=\frac{1+\sigma}{2}$ and $(\Omega^{1}_{\slashed{D}}(M),{}^*,\Psi,\pairing{\cdot}{\cdot}_{g})$ is a Hermitian differential structure. We briefly recall the necessary ingredients to deform this differential structure and refer to \cite[Section 6.1]{MRLC} for details and proofs.

Given a smooth group homomorphism $\alpha:\mathbb{T}^{2}\to \textnormal{Isom}(M,g)$, we obtain a unitary representation $U:\mathbb{T}^2\to\mathbb{B}(L^2(M,\slashed{S}))$ commuting with $\slashed{D}$ and 
such that ${\rm Ad}_U$ restricts to a group of $*$-automorphisms of $C^\infty(M)$. The representation $U$ is necessarily of the form $U(s)=e^{is_{1}p_{1}+is_{2}p_{2}}$ where the $p_{i}$ are the self-adjoint generators of the one-parameter groups associated to the coordinates $s_1,s_2$ of $\mathbb{T}^{2}$. The $*$-algebra of smooth vectors $C_{\alpha}^{\infty}(L^{2}(M,\slashed{S}))\subset\mathbb{B}(L^2(M,\slashed{S}))$ consists of elements $T$ that can be written as a norm convergent series
\[
T=\sum_{(n_{1},n_{2})\in\mathbb{Z}^{2}}T_{n_1,n_2},
\]
where the family of homogeneous components $T_{(n_1,n_2)}$ is of rapid decay.

We choose $\lambda=e^{i\theta}\in\mathbb{T}$ and define a new $*$-algebra structure on the $*$-algebra of smooth vectors $C_{\alpha}^{\infty}(L^{2}(M,\slashed{S}))\subset\mathbb{B}(L^2(M,\slashed{S}))$. 
On homogenous elements $S,T\in\mathbb{B}(L^2(M,\slashed{S}))$ with degrees $n(S)=(n_1(S),n_2(S))\in\Z^2$ and $n(T)=(n_1(T),n_2(T))\in\Z^2$, we define a new multiplication and adjoint $\dag$ via
\[
S*T=\lambda^{n_2(S)n_1(T)}ST,\qquad T^\dag=\lambda^{n_1(T)n_2(T)}T^*
\]
where $ST$ and $T^*$ are the existing composition and adjoint respectively. Extending linearly gives a new $*$-algebra structure on $C_{\alpha}^{\infty}(L^{2}(M,\slashed{S}))$, and we denote by $C^{\infty}(M_{\theta})$ the vector space $C^{\infty}(M)$ with this new $*$-algebra structure. The map defined for homogenous elements $T$ by
\[
L:C_{\alpha}^{\infty}(L^{2}(M,\slashed{S}))\to \mathbb{B}(L^{2}(M,\slashed{S})),\quad T\mapsto T\lambda^{n_{2}(T)p_{1}},
\]
 extends to a $*$-representation, and $(C^{\infty}(M_{\theta}),L^{2}(M,\slashed{S}),\slashed{D})$ is a spectral triple. 

For a pair $(S,T)$ of homogenous operators define
\begin{equation}
\label{eq: theta-cocycle}
\Theta(S,T):=\lambda^{n_{2}(S)n_{1}(T)-n_{2}(T)n_{1}(S)}=\Theta(n(S),n(T)).
\end{equation}
The map $\sigma_{\theta}:T^{2}_{\slashed{D}}(M_{\theta})\to T^{2}_{\slashed{D}}(M_{\theta})$ defined on homogeneous forms $\omega,\eta$ by
\begin{equation}
\sigma_{\theta}(\omega\otimes \eta):=\Theta(\omega,\eta)(\eta\otimes\omega),
\label{eq:sigma-theta}
\end{equation}
is a well-defined bimodule map and $\Psi_{\theta}:=\frac{1+\sigma_{\theta}}{2}$ is an idempotent that projects on the junk two-tensors. Lastly, the formula
\begin{equation}
\label{eq: theta-inner-product}
\pairing{\omega}{\eta}_{\theta}:=\lambda^{(n_{1}(\omega)-n_1(\eta))n_{2}(\omega)}\pairing{\omega}{\eta},\qquad \omega,\eta\in \Omega^1_{\slashed{D}}(M_\theta),
\end{equation}
equips $\Omega^{1}_{\slashed{D}}(M_{\theta})$ with a positive definite Hermitian inner product for which $\Psi_{\theta}$ is self-adjoint.
\begin{thm}[\cite{MRLC}, {\rm Theorem 6.12}]
\label{thm:theta-unique}
Let $M$ be a compact Riemannian manifold, $\slashed{S}\to M$ a Dirac bundle, 
$\mathbb{T}^{2}\to\textnormal{Isom}(M)$ a smooth group homomorphism and $\theta\in\mathbb{T}$. For $\Psi_{\theta}$ the $\theta$-deformed junk projection and $\pairing{\cdot}{\cdot}_{\theta}$ the $\theta$-deformed inner product,
$(\Omega^{1}_{\slashed{D}}(M_{\theta}),\dag,\Psi_{\theta},\pairing{\cdot}{\cdot}_{\theta})$ is a $\dag$-concordant Hermitian differential structure.
Moreover for 
$\sigma_{\theta}: T^{2}_{\slashed{D}}(M_{\theta})\to T^{2}_{\slashed{D}}(M_{\theta})$
the $\theta$-deformed flip map \eqref{eq:sigma-theta} there exists a unique Hermitian
  torsion-free $\dag$-bimodule connection $(\nablar^{G_{\theta}},\sigma_{\theta})$ on $\Omega^{1}_{\slashed{D}}(M_{\theta})$.
\end{thm}

In \cite{MRLC} the Levi-Civita connection $\nablar^{G_\theta}$ was constructed explicitly using an homogeneous frame for the Hermitian differential structure $(\Omega^{1}_{\slashed{D}}(M_{\theta}),\dag,\Psi_{\theta},\pairing{\cdot}{\cdot}_{\theta})$.
In subsection \ref{subsec:deformed-bimod}, we will show how to deform connections $\nablar\mapsto \nablar_\theta$  on suitable $\mathbb{T}^2$-equivariant bundles. Then in subsection \ref{subsec:LC-theta} we apply this method, and the
uniqueness guaranteed by Theorem \ref{thm:theta-unique}, to show that in fact $\nablar^{G_\theta}=\nablar^G_\theta$, where $\nablar^G$ is the Levi-Civita connection of the original manifold.

\subsection{$\theta$-deformation of inner product bimodules and connections}
\label{subsec:deformed-bimod}
We will give a general procedure for $\theta$-deformations of $\mathbb{T}^2$-equivariant bimodules over $*$-algebras $\B$, as well as connections thereon. In order to accommodate general Dirac bundles in the subsequent sections (which may or may not be $\dag$-bimodules) we work in the setting of equivariant inner product bimodules.

\begin{defn} 
Let $\B$ be a local algebra (in the $C^{*}$-algebra $B$) 
equipped with an action of $\mathbb{T}^{2}$ by $*$-automorphisms 
such that $\B$ is contained in the $C^1$-subalgebra of $B$ for the 
$\mathbb{T}^2$ action.
A $\mathbb{T}^{2}$-equivariant inner product $\B$-bimodule is 
a triple 
$(\mathcal{X},\leftindex_{\B}{\pairing{\cdot}{\cdot}},\pairing{\cdot}{\cdot}_{\B})$, 
where $\X$ is a bimodule over $\B$ equipped with a left inner 
product $\leftindex_{\B}{\pairing{\cdot}{\cdot}}$ and right inner 
product $\pairing{\cdot}{\cdot}_{\B}$ in which it becomes a 
left- and right $\mathbb{T}^{2}$-equivariant pre-Hilbert 
$C^{*}$-module over $\B$. 
\end{defn}
Setting 
$
\X^{*}:=\overrightarrow{\rm{Hom}}^{*}(\X,\B)$ and $\leftindex^{*}{\X}:=\overleftarrow{\rm{Hom}}(\X,\B),
$
the inner products define antilinear isomorphisms
\[
\X\to \X^{*},\quad x\mapsto  x^{*}:=\langle x|_{\B},\quad \X\to\leftindex^{*}{\X},\quad x\mapsto \leftindex^{*}{x}:=\leftindex_{\B}{|x\rangle}. 
\]
Given a $\dag$-bimodule $\mathcal{Y}$ over $\B$, the bimodules $\X\ox_{\B}\mathcal{Y}\ox_{\B}\X^{*}$ and $\leftindex^{*}{\X}\ox_{\B}\mathcal{Y}\ox_{\B}\X$ become $\dag$-bimodules for the operations
\[
(x_{1}\ox y\ox x_{2}^{*})^{\dag}:=x_{2}\ox y^{\dag}\ox x_{1}^{*},\quad (\leftindex^{*}{x}_{1}\ox y\ox x_{2})^{\dag}:=\leftindex^{*}{x}_{2}\ox y^{\dag}\ox x_{1}.
\]
Given a right frame $(x_{i})$ for $\X$, $(x_{i}^{*})$ is a left frame for $\X^{*}$ and a left frame $(y_{j})$ for $\X$ gives a right frame $\leftindex^{*}{y}_{j}$ for $\leftindex^{*}{\X}$.

\begin{lemma}
\label{lem: deformed-module}
Let $(\mathcal{X},\leftindex_{\B}{\pairing{\cdot}{\cdot}},\pairing{\cdot}{\cdot}_{\B})$ be a $\mathbb{T}^{2}$-equivariant inner product bimodule over the local $*$-algebra $\B$. For $a,b\in\B$ and $x,y\in \X$ all homogeneous, the formulae
\begin{align*}
a*x&:=\lambda^{n_{2}(a)n_{1}(x)}ax,  &x*b&:=\lambda^{n_{2}(x)n_{1}(b)}xb,\\
\pairing{x}{y}_{\theta}&:=\lambda^{(n_{1}(x)-n_{1}(y))n_{2}(x)}\pairing{x}{y}_\B, &\leftindex_{\theta}{\pairing{x}{y}}&:=\lambda^{(n_{2}(y)-n_{2}(x))n_{1}(y)}\leftindex_{\B}{\pairing{x}{y}}
\end{align*}
make the linear space $\X$ into a $\mathbb{T}^{2}$-equivariant inner product bimodule over $\B_\theta$, which we denote by $\X_{\theta}$. 
Here $\B_\theta$ is the deformation of $\B$ as a module over itself. The module $\X$ admits homogeneous frames and any homogeneous frame for $\X$ is a frame for $\X_{\theta}$.
\end{lemma}
\begin{proof} 
This is proved just as in \cite[Lemmas 6.4 and 6.5, Corollary 6.6]{MRLC}, where the same facts were verified for the $\theta$-deformed one-forms $\Omega^{1}_{\slashed{D}}(M_\theta)$. 
\end{proof}
{\bf Notation.} To alleviate notation, we adopt the following abbreviations for the remainder of this section. 
We  write $\ox:=\ox_\B$ and $\ox_\theta:=\ox_{\B_\theta}$. 

Given a $\mathbb{T}^{2}$-equivariant inner product right $\B$-module $\X$ and a $\mathbb{T}^{2}$-equivariant inner product bimodule $\Y$, the tensor product $\X\ox_{\B}\Y$ is an equivariant inner product right module for the action
\[
\alpha_{z}(x\ox y):=\alpha^{\X}_{z}(x)\ox\alpha^{\Y}_{z}(y),\quad z\in\mathbb{T}^{2}.
\]
Analogous statements hold for the case where $\X$ is a bimodule and $\Y$ is a left module. We prove that the interior tensor product commutes with deformation in the following sense.

\begin{lemma} 
\label{lem:tee-theta}
Let $\X,\Y$ be $\mathbb{T}^2$ equivariant $\B$-bimodules.
The map $T^{\X,\Y}_{\theta}$ defined for homogeneous elements $x,y$ by
\begin{align*}
T_{\theta}^{\X,\Y}:(\X\ox_{\B}\Y)_{\theta}\rightarrow \X_{\theta}\ox_{\B_\theta}\Y_{\theta},\quad x\ox y\mapsto \lambda^{-n_{2}(x)n_{1}(y)}x\ox_\theta y,
\end{align*}
is an isomorphism of inner product right, left or bimodules.
\end{lemma}
\begin{proof}
For homogeneous $b\in \B$ we have
\begin{align*}
T_{\theta}^{\X,\Y}(xb\ox y)
&=\lambda^{-n_{2}(x)n_{1}(y)-n_{2}(b)n_{1}(y)}xb\ox_\theta y\\
&=\lambda^{-n_{2}(x)n_{1}(y)-n_{2}(b)n_{1}(y)-n_{2}(x)n_{1}(b)}x*b\ox_\theta y\\
&=\lambda^{-n_{2}(x)n_{1}(y)-n_{2}(b)n_{1}(y)-n_{2}(x)n_{1}(b)}x\ox_\theta b*y\\
&=\lambda^{-n_{2}(x)n_{1}(y)-n_{2}(x)n_{1}(b)}x\ox_\theta by\\
&=\lambda^{-n_{2}(x)n_{1}(by)}x\ox_\theta by\\
&=T_{\theta}^{\X,\Y}(x\ox by),
\end{align*}
so $T_{\theta}^{\X,\Y}$ is compatible with the balancing relations on $\X\ox_{\B}\Y$ and $\X_\theta\ox_{\B_{\theta}}\Y_{\theta}$. 
Since $T_\theta^{\X,\Y}$ is a bilinear map on $\X\times\Y$ compatible with the balancing, it gives rise to a well-defined map on $\X\ox_\B\Y$.

Similarly, we prove that $T_{\theta}^{\X,\Y}$ preserves the inner products. Let $x_j,y_j$ be homogeneous elements of $\X,\Y$ for $j=1,2$. Then
\begin{align*}
&\pairing{T_{\theta}^{\X,\Y}(x_{1}\ox y_{1})}{T^{\X,\Y}_{\theta}(x_{2}\ox y_{2})}
=\lambda^{n_{2}(x_1)n_{1}(y_1)-n_{2}(x_{2})n_{1}(y_2)}\pairing{y_{1}}{\pairing{x_{1}}{x_{2}}_{\theta}*y_{2}}_{\theta}\\
&=\lambda^{n_{2}(x_1)n_{1}(y_1)-n_{2}(x_{2})n_{1}(y_2)}\lambda^{(n_{1}(x_1)-n_{1}(x_2))n_{2}(x_1)}\lambda^{(n_{1}(y_1)-n_{1}(y_2)+n_{1}(x_1)-n_1(x_2))n_{2}(y_1)}\!\pairing{y_{1}}{\pairing{x_{1}}{x_2}*y_{2}}\\
&=\!\lambda^{n_{2}(x_1)n_{1}(y_1)-n_{2}(x_{2})n_{1}(y_2)+(n_{1}(x_1)-n_{1}(x_2))n_{2}(x_1)+(n_{1}(y_1)-n_{1}(y_2)+n_{1}(x_1)-n_1(x_2))n_{2}(y_1)}\!\lambda^{(n_{2}(x_2)-n_{2}(x_1))n_{1}(y_1)}\\
&\qquad\qquad\qquad\qquad\qquad\qquad\qquad\qquad\qquad\qquad\qquad\qquad\qquad\qquad\qquad\times\pairing{y_{1}}{\pairing{x_{1}}{x_2}y_{2}}\\
&=\lambda^{(n_{1}(x_{1})+n_{1}(y_1)-n_{1}(x_{2})-n_{1}(y_2))(n_{2}(x_1)+n_2(y_1))}\pairing{x_1\ox y_{1}}{x_2\ox y_{2}}\\
&=\pairing{x_1\ox y_{1}}{x_2\ox y_{2}}_{\theta}.
\end{align*}
Thus $T_{\theta}^{\X,\Y}$ is an injective right module map, and as it has dense range as well, it is a unitary isomorphism.
\end{proof}

In order to study second covariant derivatives and curvature tensors, we need to be able to deform threefold tensor products.
\newcommand{\ZZ}{\mathcal{Z}}
\begin{lemma}
\label{lem:map-assoc}
Let $\X,\Y,\ZZ$ be $\mathbb{T}^2$-equivariant $\B$-bimodules.
We have an equality of linear maps 
\[
1\ox_\theta T^{\Y,\ZZ}_\theta\circ T^{\X,\Y\ox\ZZ}_\theta
=T^{\X,\Y}_\theta\ox_\theta1\circ T^{\X\ox\Y,\ZZ}_\theta:(\X\ox\Y\ox\ZZ)_\theta\to \X_\theta\ox_\theta \Y_\theta\ox_\theta\ZZ_\theta.
\] 
Denote this linear map by $H_{\theta}:(\X\ox\Y\ox\ZZ)_\theta\to \X_\theta\ox_\theta\Y_\theta\ox_\theta\ZZ_\theta$.
\end{lemma}
\begin{proof}
Evaluating both maps on a simple tensor of homogeneous $x\in\X$, $y\in\Y$ and $z\in\ZZ$ yields
\begin{align*}
1\ox_\theta T^{\Y,\ZZ}_\theta\circ T^{\X,\Y\ox\ZZ}_\theta(x\ox y\ox z)
&=1\ox_\theta T^{\Y,\ZZ}_\theta(\lambda^{-n_2(x)(n_1(y)+n_1(z))}x\ox_\theta (y\ox z))\\
&=\lambda^{-n_2(x)(n_1(y)+n_1(z))}\lambda^{-n_2(y)n_1(z)}x\ox_\theta y\ox_\theta z
\end{align*}
and
\begin{align*}
T^{\X,\Y}_\theta\ox_\theta1\circ T^{\X\ox\Y,\ZZ}_\theta(x\ox y\ox z)_\theta
&=T^{\X,\Y}_\theta\ox_\theta1(\lambda^{-(n_2(x)+n_2(y))n_1(z)}(x\ox y)\ox_\theta z)\\
&=\lambda^{-(n_2(x)+n_2(y))n_1(z)}\lambda^{-n_2(x)n_1(y)}x\ox_\theta y\ox_\theta z.
\end{align*}
Comparison of the phase factors and extending by linearity completes the proof.
\end{proof}

\begin{defn}
Let $(\B,H,\D)$ be a $\mathbb{T}^2$-equivariant spectral triple \cite{Y}, and $\Omega^1_{\D}$ the first order differential forms, 
which are equivariant for a $\mathbb{T}^2$-action. 
Given $\mathbb{T}^{2}$-equivariant right, respectively left connections
\begin{align*}
\nablar&:\mathcal{X}\to \mathcal{X}\otimes_{\B}\Omega^{1}_{\D}\quad\mbox{and}\quad
\nablal:\mathcal{X}\to \Omega^{1}_{\D}\otimes_{\B}\mathcal{X}
\end{align*}
the deformed connections are the maps
\[
\nablar_{\theta}:\mathcal{X}_{\theta}\to \mathcal{X}_{\theta}\otimes_{\B_\theta}(\Omega^{1}_{\D})_\theta,\qquad \nablar_{\theta}(x)=T^{\X,\Omega^1}_\theta(\nablar(x)),
\]
\[
\nablal_{\theta}:\mathcal{X}_{\theta}\to (\Omega^{1}_{\D})_\theta\otimes_{\B_\theta}\mathcal{X}_{\theta},\qquad\nablal_{\theta}(x)=T^{\Omega^1,\X}_\theta(\nablal(x)).
\]
\end{defn} 
It is a straightforward verification that deformed connections are indeed connections. For $x\in\X$ and $b\in\B$ we have
\begin{align*}
\nablar_{\theta}(x*b)
&=T^{\X,\Omega^1}_{\theta}(\nablar(x*b))
=\lambda^{n_{2}(x)n_{1}(b)}T^{\X,\Omega^1}_{\theta}(\nablar(xb))\\
&=\lambda^{n_{2}(x)n_{1}(b)}T^\X_{\theta}(\nablar(x)b+x\ox [\D,b])\\
&=T^{\X,\Omega^1}_{\theta}(\nablar(x))*b+\lambda^{n_{2}(x)n_{1}(b)}T^{\X,\Omega^1}_{\theta}(x\ox [\D,b])\\
&=\nablar_{\theta}(x)*b+x\ox_\theta [\D,b],
\end{align*}
and similarly for left connections. 

\subsection{Deformation of the Levi-Civita connection}
\label{subsec:LC-theta}
Let $(M,g)$ be a compact Riemannian manifold, $\slashed{S}\to M$ a $\mathbb{T}^{2}$-equivariant Dirac bundle and $(C^{\infty}(M), L^{2}(M,\slashed{S}), \slashed{D})$ the associated $\mathbb{T}^{2}$-equivariant spectral triple. Furthermore let $(\Omega^{1}_{\slashed{D}}(M),*,\Psi,\pairing{\cdot}{\cdot})$ be the Hermitian differential structure of the equivariant spectral triple $(C^{\infty}(M), L^{2}(M,\slashed{S}), \slashed{D})$, and $\nablar^G$ the associated Levi-Civita connection.

We will now show that the Levi-Civita connection $\nablar^{G_{\theta}}$ of the deformed Hermitian differential structure $(\Omega^{1}_{\slashed{D}}(M_{\theta}),\dag,\Psi_{\theta},\pairing{\cdot}{\cdot}_{\theta})$ indeed arises as the deformation of the classical Levi-Civita connection, so that $\nablar^{G_{\theta}}=\nablar^{G}_{\theta}$. We will achieve this by showing that $\nablar^{G}_{\theta}$ is an Hermitian torsion-free $\sigma_\theta$-bimodule connection, and so coincides with $\nablar^{G_\theta}$ by Theorem \ref{thm:theta-unique}.

We first consider the deformation of $\nablar^G\ox1+1\ox\nablal^\X$. This sum is a well-defined map
$\Omega^1\ox\X\to\Omega^1\ox\Omega^1\ox\X$ by \cite[Proposition 2.30]{MRLC}, and we want to compare it to  the (also well-defined)  map
\begin{align*}
&\nablar_\theta^G\ox1+1\ox\nablal^\X_\theta:\Omega^1_\theta\ox_\theta\X_\theta\to\Omega^1_\theta\ox_\theta\Omega^1_\theta\ox_\theta\X_\theta.
\end{align*}

\begin{lemma}
\label{lem:well-defined-pieces} 
With $H$ as in Lemma \ref{lem:map-assoc} there is an equality of maps
\[
\nablar^G_\theta\ox_\theta1+1\ox_\theta\nablal^\X_\theta=H_{\theta}\circ (\nablar^G\ox1+1\ox\nablal^\X)\circ (T^{\Omega^{1},\X}_{\theta})^{-1}:\Omega^{1}_{\theta}\ox_{\theta}\X_{\theta}\to \Omega^{1}_{\theta}\ox_{\theta}\Omega^{1}_{\theta}\ox_{\theta}\X_{\theta}.
\]
\end{lemma}
\begin{proof}
Using the definition of the deformed connection $\nablar^G_\theta$,  for homogeneous $\omega,x$ we have
\begin{align*}
\nablar_\theta(\omega)\ox_\theta x
&=(T^{\X,\Omega^1}_\theta\ox_{\theta}1)(\nablar(\omega)\ox_\theta x)\\
&=\lambda^{n_{2}(\omega)n_{1}(x)}(T^{\X,\Omega^1}_\theta\ox 1)\circ T^{\X\ox\Omega^1,\X}_\theta(\nablar(\omega)\ox x)\\
&=\lambda^{n_{2}(\omega)n_{1}(x)}H_{\theta}(\nablar(\omega)\ox x).
\end{align*}
Similarly we find that
\begin{align*}\omega\ox_\theta \nablal_\theta(x)=\lambda^{n_{2}(\omega)n_{1}(x)}(1\ox_\theta T^{\Omega^1,\X}_\theta)\circ T^{\X,\Omega^1\ox\X}_\theta(\omega\ox_\theta \nablal_\theta(x))=\lambda^{n_{2}(\omega)n_{1}(x)}H_{\theta}(\omega\ox \nablal(x)),
\end{align*}
which proves the claim.
\end{proof}

\begin{lemma}
\label{lem:Herm-theta}
Let $(\omega_j)\subset \Omega^1$ be an homogeneous frame,
and $G_\theta=\sum_j\omega_j\ox_\theta\omega_j^\dag\in \Omega^1\ox_{\B_\theta}\Omega^1$ the deformed metric. Then the deformation $\nablar^G_\theta$ of the Levi-Civita connection $\nablar^G$  satisfies
\[
(\nablar_\theta^G\ox1+1\ox\nablal_\theta^G)(G_\theta)=0
\]
and so $\nablar_\theta^G$ is Hermitian.
\end{lemma}
\begin{proof}
Observe that $G_{\theta}=T^{\Omega^{1},\Omega^{1}}_{\theta}(G)$, so that
\begin{align*}
(\nablar_\theta\ox1+1\ox\nablal_\theta)(G_\theta)
&=H_{\theta}(\nablar\ox1+1\ox\nablal)(G)=0.\qedhere
\end{align*}
\end{proof}

We obtain the following description of the deformed exterior derivative.
\begin{lemma} 
\label{lem:dee-theta}
Let $\omega\in\Omega^{1}_{\slashed{D}}(M)$. 
Then $\dee_{\theta}(\omega)=T_{\theta}^{\Omega^1,\Omega^1}(\dee(\omega))$. Moreover 
the braiding and junk projection satisfy $\sigma_\theta=T_{\theta}^{\Omega^1,\Omega^1}\circ\sigma\circ(T_{\theta}^{\Omega^1,\Omega^1})^{-1}$ and $\Psi_\theta=T_{\theta}^{\Omega^1,\Omega^1}\circ\Psi\circ(T_{\theta}^{\Omega^1,\Omega^1})^{-1}$.
\end{lemma}
\begin{proof}
For all $\theta$, the exterior derivative of a form $a*[\D,b]$ with $a,b$ homogenous is given by
\begin{align*}
d_{\theta}(a*[\D,b])=(1-\Psi_{\theta})[\D,a]\ox[\D,b]=\frac{1}{2}\left([\D,a]\ox[\D,b]-\Theta(a,b)[\D,b]\ox[\D,a]\right).
\end{align*}
Now since $a[\D,b]=\lambda^{-n_{2}(a)n_{1}(b)}a*[\D,b]$ and $\lambda^{-n_{2}(a)n_{1}(b)}\Theta(a,b)=\lambda^{-n_{2}(b)n_{1}(a)}$ we obtain
\[
\dee_{\theta}(a[\D,b])=\frac{1}{2}\left(\lambda^{-n_{2}(a)n_{1}(b)}[\D,a]\ox[\D,b]-\lambda^{-n_{2}(b)n_{1}(a)}[\D,b]\ox[\D,a]\right),
\]
and extension by linearity then gives the asserted formula.
The relations for  $\sigma$ and $\Psi$ are straightforward verifications.
\end{proof}

\begin{lemma}
\label{lem:TF-theta}
Let $\nablar^G$ be the undeformed Levi-Civita connection. 
The deformed connection $\nablar^G_\theta$ is torsion-free.
\end{lemma}
\begin{proof}
Lemma \ref{lem:Herm-theta} shows that the deformed connections are Hermitian.
Writing $T_\theta=T^{\Omega^1,\Omega^1}_\theta$ and recalling that $\Psi_\theta=T_\theta\circ\Psi\circ T_\theta^{-1}$, we can use Lemma \ref{lem:dee-theta} and the torsion-free property for $\nablal^G$ to see that
\[
\dee_\theta=T_\theta\circ\d
=T_\theta\circ(1-\Psi)\circ\nablal^G=(1-\Psi_\theta)\nablal_\theta^G
\]
whence $\nablar_\theta$ is torsion-free.
\end{proof}

\begin{lemma}
\label{lem:dag-theta}
The deformed adjoints $\dag_{1\ox_\theta1}$ on  $\Omega^1_\theta\ox_{\theta}\Omega^1_\theta$ and $\dag_{2,\theta}$ on $(\Omega^1\ox\Omega^1)_\theta$ respectively are
related by $\dag_{1\ox_\theta1}= T_\theta\circ \dag_{2,\theta}\circ  T_\theta^{-1}$
where $T_\theta=T_\theta^{\Omega^1,\Omega^1}$.
\end{lemma}
\begin{proof}
We check that the diagram 
\[
\xymatrix{\Omega^1_\theta\ox_{\theta}\Omega^1_\theta\ar[rr]^{ T_\theta^{-1}}\ar[d]_{\dag_{1\ox_\theta1}}&& (\Omega^1\ox\Omega^1)_\theta\ar[d]_{\dag_{2,\theta}}\\
\Omega^1_\theta\ox_{\theta}\Omega^1_\theta\ar[rr]_{ T_\theta^{-1}} &&(\Omega^1\ox\Omega^1)_\theta}
\]
commutes by computing on elementary tensors of homogeneous one-forms. Recalling that 
\[
\dag_{2,\theta}(\omega\ox\rho)=\lambda^{(n_2(\omega)+n_2(\rho))(n_1(\omega)+n_1(\rho))}\rho^*\ox\omega^*
=\lambda^{(n_2(\omega)+n_2(\rho))(n_1(\omega)+n_1(\rho))}\dag_2(\omega\ox\rho),
\] 
we have
\[
\xymatrix{\omega\ox_\theta\rho\ar[rr]^{ T_\theta^{-1}}\ar[d]_{\dag_{1\ox_\theta1}}&& \lambda^{n_2(\omega)n_1(\rho)}\omega\ox\rho\ar[d]_{\dag_{2,\theta}}\\
\lambda^{n_2(\rho)n_1(\rho)+n_2(\omega)n_1(\omega)}\rho^*\ox_\theta\omega^*\ar[rr]_{ T_\theta^{-1}\qquad} &&\lambda^{-n_2(\omega)n_1(\rho)+(n_2(\rho)+n_2(\omega))(n_1(\rho)+n_1(\omega))}\rho^*\ox\omega^*}
\]
So $\dag_{1\ox_\theta1}= T_\theta\circ \dag_{2,\theta}\circ  T_\theta^{-1}$
\end{proof}

\begin{thm}
\label{prop:deformed-uniqueness}
The Levi-Civita connection $\nablar^{G_\theta}$ for $C^{\infty}(M_\theta)$
is the deformation $\nablar^G_\theta$ of the Levi-Civita connection $\nablar^G$ for $C^{\infty}(M)$.
\end{thm}
\begin{proof}
We make use of Theorem \ref{thm:theta-unique} and prove that the deformed connection 
$\nablar^G_\theta$ is Hermitian, torsion-free and a $\dag$-bimodule connection for the braiding $\sigma_\theta$. By uniqueness of the Levi-Civita connection, this will show that $\nablar^{G_\theta}=\nablar^G_\theta$. Lemmas \ref{lem:Herm-theta} and \ref{lem:TF-theta} showed that $\nablar^G_\theta$ is Hermitian and torsion-free, so we need only prove that it is a $\dag$-bimodule connection.

To show that $\nablar^G_\theta$ is a $\dag$-bimodule connection, it suffices to prove that
\[
\nablar_\theta^G(\omega)=-\sigma_\theta\circ \dag_{1\ox_\theta1}\circ\nablar^G_\theta(\omega^\dag).
\]
Using Lemma \ref{lem:dag-theta} and (writing $T_\theta=T_\theta^{\Omega^1,\Omega^1}$) $\sigma_\theta=T_\theta\circ\sigma\circ T_\theta^{-1}$ we have 
\begin{align*}
-\sigma_\theta\circ \dag_{1\ox_\theta1}\circ\nablar^G_\theta\circ\dag_{1,\theta}(\omega)
&=-\sigma_\theta\circ \dag_{2,\theta}\lambda^{n_2(\omega)n_1(\omega)}\circ\nablar^G_\theta(\omega^*)\\
&=-T_\theta\circ \sigma\circ \lambda^{-n_2(\omega)n_1(\omega)}\dag_{2,\theta}\circ \nablar^G(\omega^*)\\
&=-T_\theta( \sigma\circ \lambda^{-n_2(\omega)n_1(\omega)}\lambda^{n_2(\omega)n_1(\omega)}\dag_{2}\circ \nablar^G(\omega^*))\\
&=-T_\theta( \sigma\circ \dag_{2}\circ \nablar^G(\omega^*))\\
&=T_\theta(\nablar^G(\omega))\\
&=\nablar^G_\theta(\omega),
\end{align*}
since $\nablar^G$ is a $\dag$-bimodule connection. Hence $\nablar^G_\theta$ is the unique Hermitian torsion-free $\dag$-bimodule connection, and hence agrees with $\nablar^{G_\theta}$.
\end{proof}

In the case of free actions of $\mathbb{T}^2$, \cite{BGJ2,BGJ1} defined the Levi-Civita connection directly as the deformed connection and proved existence and uniqueness. Hence
Theorem \ref{prop:deformed-uniqueness} shows that our Levi-Civita connection agrees with theirs. 

\subsection{Invariance of the scalar curvature}

We now proceed to show that the scalar curvature $r_{\theta}\in C^{\infty}(M_{\theta})$ remains undeformed in the sense that $r_{\theta}=r_{0}=r$ in the linear space $C^{\infty}(M)$. We also show that the full curvature tensor and Ricci tensor transform naturally under deformation.
First we record a lemma about contractions with $G_\theta$.
\begin{lemma}
\label{lem:contraction-4}
Let $\omega,\rho,\eta,\tau\in \Omega^1$ be homogeneous with $n(\omega)+n(\rho)+n(\eta)+n(\tau)=0$. Then
\[
\leftindex_{\theta}{\pairing{H_\theta(\omega\ox\rho\ox\eta)\ox_\theta\tau}{G_\theta}}=T_\theta^{\Omega^1,\Omega^1}(\omega\ox\rho)\leftindex_{0}{\pairing{\eta\ox\tau}{G}}.
\]
\end{lemma}
\begin{proof}
This is just a computation. With $\omega,\rho,\eta,\tau$ as in the statement we have
\begin{align*}
\leftindex_{\theta}{\pairing{H_\theta(\omega\ox\rho\ox\eta)\ox_\theta\tau}{G_\theta}}
&=\leftindex_{\theta}{\pairing{T^{\Omega^1,\Omega^1}_\theta\circ T^{\Omega^1\ox\Omega^1,\Omega^1}_\theta(\omega\ox\rho\ox\eta)\ox_\theta\tau}{G_\theta}}\\
&=\leftindex_{\theta}{\pairing{T^{\Omega^1,\Omega^1}_\theta(\omega\ox\rho)\ox_\theta\eta\ox_\theta\tau}{G_\theta}}\lambda^{-(n_2(\omega)+n_2(\rho))n_1(\eta)}\\
&=T^{\Omega^1,\Omega^1}_\theta\big((\omega\ox\rho)*\leftindex_{\theta}{\pairing{\eta\ox_\theta\tau^\dag}{G_\theta}}\big)\lambda^{-(n_2(\omega)+n_2(\rho))n_1(\eta)}\\
&=T^{\Omega^1,\Omega^1}_\theta\big((\omega\ox\rho)\leftindex_{\theta}{\pairing{\eta\ox_\theta\tau}{G_\theta}}\big)\lambda^{(n_2(\omega)+n_2(\rho))n_1(\tau)}\\
&=T^{\Omega^1,\Omega^1}_\theta\big((\omega\ox\rho)\leftindex_{\theta}{\pairing{\eta}{\tau^\dag}}\big)\lambda^{(n_2(\omega)+n_2(\rho))n_1(\tau)}\\
&=T^{\Omega^1,\Omega^1}_\theta\big((\omega\ox\rho)\leftindex_{0}{\pairing{\eta}{\tau^\dag}}\big)\lambda^{(n_2(\omega)+n_2(\rho)+n_2(\tau)+n_2(\eta))n_1(\tau)}\\
&=T_\theta^{\Omega^1,\Omega^1}(\omega\ox\rho)\leftindex_{0}{\pairing{\eta\ox\tau}{G}}
\end{align*}
where in the last step we used the assumption on the degrees of the one-forms.
\end{proof}

\begin{thm}
\label{thm:curv-theta} 
For a one-form $\rho\in\Omega^1_{\slashed{D}}(M)=\Omega^1_{\slashed{D}}(M_\theta)$ we have
\[
R^{\nablal^{G_\theta}}(\rho)=H_{\theta}(R^{\nablal^G}(\rho)),
\quad\quad
{\rm Ric}^{\nablal^{G_\theta}}=T^{\Omega^1,\Omega^1}_\theta({\rm Ric}^{\nablal^G})
\quad\mbox{and}\quad
r^{\nablal^{G_\theta}}=r^{\nablal^G}.
\]
\end{thm}
\begin{proof}
We saw in Lemma \ref{lem:well-defined-pieces} that
\[
(\nablar^{G_\theta}\ox1+1\ox\nablal^{G_\theta})\circ\nablal^{G_\theta}=H_{\theta}((\nablar^{G}\ox1+1\ox\nablal^{G})\circ\nablal^{G})_\theta).
\]
Since $(1-P_\theta)=(T^{\Omega^1,\Omega^1}_\theta(1-\Psi)(T^{\Omega^1,\Omega^1}_\theta)^{-1})\ox1$, one now checks directly that
\begin{align*}
R^{\nablal^{G_\theta}}(\rho)&=(1-P_\theta)(\nablar^{G_\theta}\ox1+1\ox\nablal^{G_\theta})\circ\nablal^{G_\theta}(\rho)\\
&=H_\theta((1-P)(\nablar^{G}\ox1+1\ox\nablal^{G})\circ\nablal^{G}(\rho))\\
&=H_\theta(R^{\nablal^G}(\rho)).
\end{align*}

For the Ricci curvature, we write 
\[
R^{\nablal^{G_\theta}}(\omega_{j})\ox_\theta\omega_j^\dag=H_\theta(R(\omega_j))\ox_\theta\omega_j^\dag.
\]
Since $\nablar^G,\nablal^G,\d$ are all degree zero, we see that $R^{\nablal^{G_\theta}}(\omega_{j})\ox_\theta\omega_j^\dag$ is degree zero.
We can now compute the Ricci curvature of $\nablal^{G_\theta}$ using Lemma \ref{lem:contraction-4}
\begin{align*}
{\rm Ric}^{\nablal^{G}_\theta}
&=\leftindex_{\theta}{\pairing{R^{\nablal^{G}_\theta}}{G_\theta}}
=\leftindex_{\theta}{\pairing{H_\theta(R^{\nablal^{G}}(\omega_j))\ox_\theta\omega_j^\dag}{G_\theta}}
=T_\theta^{\Omega^1,\Omega^1}({\rm Ric}^{\nablal^G}).
\end{align*}
Finally the scalar curvature is given by
\[
r^{\nablal^{G_\theta}}
=\pairing{G_\theta}{{\rm Ric}^{\nablal^{G_\theta}}}_\theta
=\pairing{T_\theta^{\Omega^1,\Omega^1}(G)}{T_\theta^{\Omega^1,\Omega^1}({\rm Ric}^{\nablal^G})}_\theta
=\pairing{G}{{\rm Ric}^{\nablal^G}}_0=r^{\nabla^G}.\qedhere
\]
\end{proof}

\subsection{Dirac spectral triple and Weitzenbock formula}
To establish the Weitzenbock formula for $\theta$-deformations of manifolds, we now consider the deformation of Clifford connections. 

Recall that for a $\theta$-deformed Dirac bundle we have the maps
$m_{\theta}:T^{2}_{\slashed{D}}(M_{\theta})\to \mathbb{B}(L^{2}(M,\slashed{S}))$, $c_{\theta}:\Omega^{1}_{\slashed{D}}(M)\ox_{C^{\infty}(M)}\Gamma(M,\slashed{S})\to \Gamma(M,\slashed{S})$, $\sigma_\theta:T^2_{\slashed{D}}(M_\theta)\to T^2_{\slashed{D}}(M_\theta)$
and  $g_{\theta}:T^{2}_{\slashed{D}}(M_{\theta})\to C^{\infty}(M_{\theta})$. For later computations we describe how these maps interact with $T_\theta$ and $H_\theta$ of Lemmas \ref{lem:tee-theta} and \ref{lem:map-assoc}.

\begin{lemma}
\label{lem: invariant-c}
Let $\slashed{S}\to M$ be a $\mathbb{T}^{2}$-equivariant Dirac bundle over the compact Riemannian manifold $(M,g)$, $(C^{\infty}(M),L^{2}(M,\slashed{S}), \slashed{D})$ the associated $\mathbb{T}^{2}$-equivariant Dirac spectral triple. 
  With $T^{\Omega^{1},\Omega^{1}}_{\theta},T^{\Omega^{1},\X}_{\theta}$ and $H_{\theta}=H^{\Omega^{1},\Omega^{1},\X}_{\theta}$ as in Lemmas \ref{lem:tee-theta} and \ref{lem:map-assoc} we have
\begin{enumerate}[noitemsep]
\item $m=m_{\theta}\circ T_{\theta}^{\Omega^{1},\Omega^{1}}$;
\item $g=g_{\theta}\circ T_{\theta}^{\Omega^{1},\Omega^{1}}$;
\item $c=c_{\theta}\circ T^{\Omega^{1},\X}_{\theta}$;
\item $(1\ox_{\theta} c_{\theta})\circ H_{\theta}=T^{\Omega^{1},\X_{\theta}}\circ (1\ox c)$;
\item $\sigma_{\theta}\ox_{\theta} 1=H_{\theta}(\sigma\ox 1)H_{\theta}^{-1}$.
\end{enumerate}
\end{lemma}
\begin{proof} These are all straightforward verifications using the definitions.
\end{proof}

\begin{prop} The $\theta$-deformed Clifford connection condition
\begin{equation}
\label{eq: theta-Clifford-connection}
\nablal^{\slashed{S}}_{\theta}\circ c_{\theta}=(1\ox_{\theta}c_{\theta})(\sigma_{\theta}\ox_{\theta}1)(1\ox_{\theta}\nablal^{\slashed{S}}_{\theta}+\nablar^{G_{\theta}}\ox_{\theta}1),
\end{equation}
holds on the module $\X$.
\end{prop}
\begin{proof} 
Using Lemma \ref{lem: invariant-c} we compute and compare
\begin{align*}
\nablal^{\slashed{S}}_{\theta}\circ c_{\theta}=(T^{\Omega^{1},\X}_{\theta})(\nablal^{\slashed{S}}\circ c)(T^{\Omega^{1},\X}_{\theta})^{-1},
\end{align*}
and 
\begin{align*}
(1\ox_{\theta}c_{\theta})(\sigma_{\theta}\ox_{\theta}1)&(1\ox_{\theta}\nablal^{\slashed{S}}_{\theta}+\nablar^{G_{\theta}}\ox_{\theta}1)\\
&=(T^{\Omega^{1},\X}_{\theta})(1\ox c)H_{\theta}^{-1}H_{\theta}(\sigma\ox 1)H_{\theta}^{-1}H_{\theta}(1\ox\nablal^{\slashed{S}}+\nablar^{G}\ox 1)(T^{\Omega^{1},\X}_{\theta})^{-1}\\
&=(T^{\Omega^{1},\X}_{\theta})(1\ox c)(\sigma\ox 1)(1\ox\nablal^{\slashed{S}}+\nablar^{G}\ox 1)(T^{\Omega^{1},\X}_{\theta})^{-1}.
\end{align*}
Since 
\[
\nablal^{\slashed{S}}\circ c=(1\ox c)(\sigma\ox 1)(1\ox\nablal^{\slashed{S}}+\nablar^{G}\ox 1),
\]
the statement follows.
\end{proof}
\begin{prop}
\label{prop:deformed-dirac-triple}
Let $\slashed{S}\to M$ be a $\mathbb{T}^{2}$-equivariant Dirac bundle. 
The connections $\nablar^{G_{\theta}}$ and $\nablal^{\slashed{S}}_{\theta}$ make $(C^{\infty}(M_{\theta}), L^{2}(M,\slashed{S}),\slashed{D})$
into a Dirac spectral triple over $(C^{\infty}(M_{\theta}),\dag,\Psi_{\theta},\pairing{\cdot}{\cdot}_{\theta})$.
\end{prop} 
\begin{proof}
We start with 
Condition 1 of Definition \ref{def: Dirac-spectral-triple}. We have that $G_{\theta}=T^{\Omega^{1},\Omega^{1}}_{\theta}(G)$, so by Lemma \ref{lem: invariant-c} we have 
\begin{align}
\label{eq: invariant-m-G}
m_{\theta}(G_{\theta})=m_{\theta}\circ T^{\Omega^{1},\Omega^{1}}_{\theta}(G)=m(G) =\dim M\,{\rm Id}_{\slashed{S}}.
\end{align}
Using Equation \eqref{eq:ee-beta} and Lemma \ref{lem: invariant-c} we have
\begin{align}
\label{eq: invariant-e-beta}
e^{\beta_{\theta}}=-g_\theta(G_\theta)=-g_\theta(T_\theta(G))=-g(G)=e^\beta.
\end{align}
Given $\rho,\eta\in \Omega^1_{\slashed{D}}(M_\theta)$ we have 
\begin{align*}
m_{\theta}\circ\Psi_{\theta}&=m_{\theta}\circ T^{\Omega^{1},\Omega^{1}}_{\theta}\circ\Psi\circ(T^{\Omega^{1},\Omega^{1}}_{\theta})^{-1}=m\circ \Psi\circ (T^{\Omega^{1},\Omega^{1}}_{\theta})^{-1}=g\circ (T^{\Omega^{1},\Omega^{1}}_{\theta})^{-1}=g_{\theta}.
\end{align*}
Condition 2 holds since $c:\Omega^{1}\ox\X\to \X$ and $c_{\theta}=c\circ T^{\Omega^{1},\X}_{\theta}$ so that 
\[c_{\theta}:\Omega^{1}_{\theta}\ox_{\theta}\X_\theta\to \X_{\theta}.\]
For Condition 3, we have 
\[\slashed{D}=c\circ\nablal^{\slashed{S}}=c_{\theta}\circ T^{\Omega^{1},\X}_{\theta}\circ \nablal^{\slashed{S}}=c_{\theta}\circ\nablal^{\slashed{S}}_{\theta}.\]
Condition 4 follows by applying $c_{\theta}$ to Equation \ref{eq: theta-Clifford-connection}.
\end{proof}

\begin{thm} 
Let $\slashed{S}\to M$ be a $\mathbb{T}^{2}$-equivariant Dirac bundle over a compact Riemannian manifold $(M,g)$ and $(C^{\infty}(M_{\theta}),L^{2}(M,\slashed{S}),\slashed{D})$ an associated $\theta$-deformed spectral triple. Then the connection $\nablal^{\slashed{S}}_{\theta}$ satisfies the Weitzenbock formula
\[
\slashed{D}^{2}-\Delta^{\slashed{S}}_{\theta}=c_{\theta}\circ (m_{\theta}\circ\sigma_{\theta}\ox 1)(R^{\nablal^{\slashed{S}}_{\theta}}).
\]
\end{thm}
\begin{proof}
In view of Theorem \ref{thm:Weitz1} and Proposition \ref{prop:deformed-dirac-triple}, we need only verify that $m_{\theta}\circ\sigma_{\theta}\circ\Psi_{\theta}=m_{\theta}\circ\Psi_{\theta}$ and $\Psi_{\theta}(G_{\theta})=G_{\theta}$. Since $\sigma_{\theta}=2\Psi_{\theta}-1$ the first condition is immediate. 
Using Lemma \ref{lem: invariant-c} again yields
\begin{align*}
\sigma_{\theta}(G_{\theta})=T_\theta\sigma T_\theta^{-1}(T_\theta(G))=T_\theta(\sigma(G))=T_\theta(G)=G_\theta,
\end{align*}
which completes the proof.
\end{proof}
We prove another result about contractions with $G_\theta$.
\begin{lemma}
\label{lem:contraction-3}
For homogeneous $\omega,\rho\in\Omega^1$ and $x\in\X$ we have 
\[
\pairing{G_\theta}{H_\theta^{\Omega^1,\Omega^1,\X}(\omega\ox\rho\ox x)}_{\X_\theta}
=\pairing{G}{\omega\ox\rho}_\X x=\pairing{G}{\omega\ox\rho\ox x}_\X .
\]
\end{lemma}
\begin{proof}
We compute using the definitions of the deformed inner product and multiplication from Lemma \ref{lem: deformed-module}, and the maps $T_\theta$ and $H_\theta$ from Lemmas \ref{lem:tee-theta} and \ref{lem:map-assoc} to find
\begin{align*}
\pairing{G_\theta}{H_\theta^{\Omega^1,\Omega^1,\X}(\omega\ox\rho\ox x)}_{\X_\theta}
&=\pairing{T^{\Omega^1,\Omega^1}_\theta(G)}{T^{\Omega^1,\Omega^1}_\theta\ox_\theta1\circ T_\theta^{\Omega^1\ox\Omega^1,\X}(\omega\ox\rho\ox x)}_{\X_\theta}\\
&=\lambda^{-(n_2(\omega)+n_2(\rho))n_1(x)}\pairing{T^{\Omega^1,\Omega^1}_\theta(G)}{T^{\Omega^1,\Omega^1}_\theta\ox_\theta1(\omega\ox\rho)\ox_\theta x}_{\X_\theta}\\
&=\lambda^{-(n_2(\omega)+n_2(\rho))n_1(x)}\pairing{T^{\Omega^1,\Omega^1}_\theta(G)}{T^{\Omega^1,\Omega^1}_\theta(\omega\ox\rho)}_{\X_\theta} * x\\
&=\pairing{T^{\Omega^1,\Omega^1}_{\X_\theta}(G)}{T^{\Omega^1,\Omega^1}_\theta(\omega\ox\rho)}_{\X_\theta}  x\\
&=\pairing{G}{\omega\ox\rho}_\X x=\pairing{G}{\omega\ox\rho\ox x}_\X 
\end{align*}
as claimed.
\end{proof}

\begin{prop} 
Let $\slashed{S}\to M$ be a $\mathbb{T}^{2}$-equivariant Dirac bundle over a compact Riemannian manifold $(M,g)$, $\X=\Gamma(M,\slashed{S})$ and $(C^{\infty}(M_{\theta}),L^{2}(M,\slashed{S}),\slashed{D})$ the associated $\theta$-deformed spectral triple. The connection Laplacian $\Delta^{\slashed{S}}_{\theta}:\Gamma(M,\slashed{S})\to L^{2}(M,\slashed{S})$ remains undeformed, that is $\Delta^{\slashed{S}}_{\theta}=\Delta^{\slashed{S}}$.
\end{prop}
\begin{proof}
First consider the map
$
(\nablar_\theta^{G_\theta}\ox_\theta1+1\ox_\theta\nablal^{\slashed{S}}_\theta)\circ\nablal^{\slashed{S}}_\theta:\X_\theta\to \Omega^1_\theta\ox_\theta\Omega^1_\theta\ox_\theta\X_\theta,
$
and recall that
\begin{align*}
(\nablar_\theta^{G}\ox_\theta1+1\ox_\theta\nablal^\slashed{S}_\theta)\circ\nablal^\slashed{S}_\theta(x)
&=H_{\theta}(\nablar^{G}\ox1+1\ox\nablal^{\slashed{S}})(T_\theta)^{-1}\circ T_\theta(\nablal^{\slashed{S}}(x))\\
&=H_{\theta}((\nablar^{G}\ox1+1\ox\nablal^{\slashed{S}})\circ\nablal^{\slashed{S}}(x)).
\end{align*}
To obtain the connection Laplacian for $\X$, we  contract  with $G_\theta=\sum_jT^{\Omega^1,\Omega^1}_\theta(\omega_j\ox\omega_j^*)$.
Using Lemma \ref{lem:contraction-3} we have
\begin{align*}
\pairing{G_\theta}{(\nablar_\theta^{G_\theta}\ox_\theta1+1\ox_\theta\nablal^\slashed{S}_\theta)\circ\nablal^\slashed{S}_\theta(x)}_{\X_\theta}
&=\pairing{G_\theta}{H^{\Omega^1,\Omega^1,\X}_\theta(\nablar^G\ox1+1\ox\nablal^{\slashed{S}})\circ\nablal^{\slashed{S}}(x)}_{\X_\theta}\\
&=\pairing{G}{(\nablar^{G}\ox1+1\ox\nablal^{\slashed{S}})\circ\nablal^{\X}(x)}_\X.
\end{align*}
Now since $e^{-\beta_\theta}m(G_\theta)=e^{-\beta}m(G)$, we have
\begin{align*}
\Delta^{\slashed{S}}_\theta(x)
&=e^{-\beta_\theta}m(G_\theta)\pairing{G_\theta}{(\nablar_\theta^{G_\theta}\ox_\theta1+1\ox_\theta\nablal^{\slashed{S}}_\theta)\circ\nablal^{\slashed{S}}_\theta(x)}_{\X_\theta}\\
&=e^{-\beta}m(G)\pairing{G}{(\nablar^{G}\ox1+1\ox\nablal^{\slashed{S}})\circ\nablal^{\slashed{S}}(x)}_\X\\
&=\Delta^{\slashed{S}}(x).\qedhere
\end{align*}
\end{proof}

\begin{corl}
\label{cor:cliff-R-inv}
Let $\slashed{S}$ be a $\mathbb{T}^{2}$-equivariant Dirac bundle over a Riemannian spin manifold $(M,g)$ and $(C^{\infty}(M_{\theta}),L^{2}(M,\slashed{S}),\slashed{D})$ an associated $\theta$-deformed spectral triple. Then the Clifford representation of the curvature of $\nabla^{\slashed{S}}_{\theta}$ remains undeformed, that is 
$$
c_{\theta}\circ((m_{\theta}\circ\sigma_{\theta})\ox 1)(R^{\nablal^{\slashed{S}}_{\theta}})=c\circ((m\circ\sigma)\ox 1)(R^{\nablal^{\slashed{S}}}).
$$

In particular, if $\slashed{S}$ is the spinor bundle of a manifold, the Lichnerowicz formula says that $c_{\theta}\circ((m_{\theta}\circ\sigma_{\theta})\ox 1)(R^{\nablal^{\slashed{S}}_{\theta}})=r_\theta/4=r/4$ as elements of $C(M)$. 
\end{corl}
\begin{proof}
This follows from the invariance of $\slashed{D}^2-\Delta^{\slashed{S}}$ and Theorem \ref{thm:curv-theta}.
\end{proof}

\end{document}